\documentclass[11pt]{aims}
\usepackage{amsmath, amssymb, mathrsfs}
  \usepackage{paralist}
  \usepackage{graphics} %% add this and next lines if pictures should be in esp format
  \usepackage{epsfig} %For pictures: screened artwork should be set up with an 85 or 100 line screen
\usepackage{graphicx}  \usepackage{epstopdf}%This is to transfer .eps figure to .pdf figure; please compile your paper using PDFLeTex or PDFTeXify.
 \usepackage[colorlinks=true]{hyperref}
 \hypersetup{urlcolor=blue, citecolor=red}
 \usepackage[toc,page]{appendix}
\usepackage{chngcntr}

\usepackage{titlesec}
\titleformat{\section}{\Large\bfseries}{\thesection.}{4pt}{}
\titleformat{\subsection}{\large\bfseries}{\thesection.\arabic{subsection}.}{4pt}{}
\titleformat{\subsubsection}{\bfseries}{\thesection.\arabic{subsection}.\arabic{subsubsection}.}{4pt}{}
\titleformat*{\paragraph}{\bfseries}
\titleformat*{\subparagraph}{\bfseries}
\usepackage[margin=1.1in]{geometry}

  \def\currentyear{}
   \def\currentmonth{}
    \def\ppages{}
     \def\DOI{}

\newtheorem{theorem}{Theorem}[section]
\newtheorem{corollary}[theorem]{Corollary}

\newtheorem{lemma}[theorem]{Lemma}
\newtheorem{proposition}[theorem]{Proposition}
\theoremstyle{definition}

\newtheorem{remark}[theorem]{Remark}

\newcommand{\RN}{\mathbb{R}^N}
\newcommand{\R}{\mathbb{R}}
\newcommand{\N}{\mathbb{N}}
\newcommand{\Lc}{\mathscr{L}}

\numberwithin{equation}{section}

\title[Refined regularity of the blow-up set] %use the shortened version of the full title
      {Refined regularity of the blow-up set linked to refined asymptotic behavior for the semilinear heat equation}

\author[T. Ghoul, V. T. Nguyen, H. Zaag]{}
\subjclass{Primary: 35K50, 35B40; Secondary: 35K55, 35K57.}
 \keywords{Blowup solution, Blowup set, Blowup profile, Regularity, Semilinear heat equation}

 \email[T. Ghoul]{teg6@nyu.edu}
 \email[V. T. Nguyen]{Tien.Nguyen@nyu.edu}
 \email[H. Zaag]{Hatem.Zaag@univ-paris13.fr}

\thanks{H. Zaag is supported by the ERC Advanced Grant no. 291214, BLOWDISOL and by the ANR project ANA\'E ref. ANR-13-BS01-0010-03. \\ -----------------\\ \today}
\begin{document}
\maketitle

% Enter the first author's name and address:
\centerline{\scshape Tej-Eddine Ghoul$^\dagger$, Van Tien Nguyen$^\dagger$ and Hatem Zaag$^\ast$}
\medskip
{\footnotesize
 \centerline{$^\dagger$New York University in Abu Dhabi, P.O. Box 129188, Abu Dhabi, United Arab Emirates.}
  \centerline{$^\ast$Universit\'e Paris 13, Sorbonne Paris Cit\'e, LAGA, CNRS (UMR 7539), F-93430, Villetaneuse, France.}
}

\bigskip

\begin{abstract} We consider $u(x,t)$, a solution of $\partial_tu = \Delta u + |u|^{p-1}u$ which blows up at some time $T > 0$, where $u:\RN \times[0,T) \to \R$, $p > 1$ and $(N-2)p < N+2$. Define $S \subset \RN$ to be the blow-up set of $u$, that is the set of all blow-up points. Under suitable nondegeneracy conditions, we show that if $S$ contains a $(N-\ell)$-dimensional continuum for some $\ell \in \{1,\dots, N-1\}$, then $S$ is in fact a $\mathcal{C}^2$ manifold. The crucial step is to derive a refined asymptotic behavior of $u$ near blow-up. In order to obtain such a refined behavior, we have to abandon the  explicit profile function as a first order approximation and take a non-explicit function as a first order description of the singular behavior. This way we escape logarithmic scales of the variable $(T-t)$ and reach significant small terms in the polynomial order $(T-t)^\mu$ for some $\mu > 0$. The refined asymptotic behavior yields geometric constraints of the blow-up set, leading to more regularity on $S$.
\end{abstract}

\maketitle
\section{Introduction.}
We are interested in the following semilinear heat equation: 
\begin{equation}\label{equ:problem}
\left\{
\begin{array}{rcl}
\partial_t u &=& \Delta u + |u|^{p-1}u, \\
u(0) &=& u_0 \in L^\infty(\mathbb{R}^N),
\end{array}
\right.
\end{equation}
where $u(t): x \in \mathbb{R}^N \to u(x,t) \in \mathbb{R}$, $\Delta$ denotes the Laplacian in $\mathbb{R}^N$, and $p > 1$ or $1 < p < \frac{N + 2}{N - 2}$ if $N \geq 3$. 

It is well known that for each initial data $u_0$ the Cauchy problem \eqref{equ:problem} has a unique solution $u \in \mathcal{C}([0,T), L^\infty(\RN))$ for some $0 < T \leq +\infty$, and that either $T = +\infty$ or 
$$T < +\infty \quad \text{and} \quad \lim_{t \to T}\|u(t)\|_{L^\infty} = +\infty.$$
In the latter case we say that the solution blows up in finite time, and $T$ is called the blow-up time. In such a blow-up case, a point $\hat{a} \in \RN$ is called a blow-up point if $u(x,t)$ is not locally bounded in some neighborhood of $(\hat{a},T)$, this means that there exists $(x_n, t_n) \to (\hat{a}, T)$ such that $|u(x_n,t_n)| \to +\infty$ when $n \to +\infty$. We denote by $S$ the blow-up set, that is the set of all blow-up points of $u$.

Given $\hat{a} \in S$, we know from Vel\'azquez \cite{VELcpde92} (see also Filippas and Kohn \cite{FKcpam92}, Filippas and Liu \cite{FLaihn93}, Herrero and Vel\'azquez \cite{HVaihn93}, Merle and Zaag \cite{MZma00}) that up to replacing $u$ by $-u$, one of 
following two cases occurs:

- Case 1 (non degenerate rate of blow-up): For all $K_0 > 0$, there is an orthonormal $(N\times N)$ matrix $Q_{\hat{a}}$ and $\ell_{\hat{a}} \in \{1, \dots, N\}$ such that 
\begin{equation}\label{equ:case1}
\sup_{|\xi| \leq K_0} \left|(T-t)^\frac{1}{p-1}u(\hat{a} + Q_{\hat{a}}\xi\sqrt{(T-t)|\log (T-t)|}, t) - f_{\ell_{\hat{a}}}(\xi)\right| \to 0 \;\; \text{as}\;\; t \to T,
\end{equation}
where 
\begin{equation}\label{def:fl}
f_{\ell_{\hat{a}}}(\xi) = \left(p-1 + \frac{(p-1)^2}{4p} \sum_{i = 1}^{\ell_{\hat{a}}}\xi_i^2 \right)^{-\frac{1}{p-1}}.
\end{equation}

- Case 2 (degenerate rate of blow-up): For all $K_0 \geq 0$, there exists an even integer $m \geq 4$ such that 
\begin{equation}\label{equ:case2}
\sup_{|\xi| \leq K_0}\left|(T-t)^\frac{1}{p-1}u(\hat{a} + \xi(T-t)^\frac{1}{m},t) - \left(p-1 + \sum_{|\alpha| = m} c_\alpha\xi^\alpha\right)^{-\frac{1}{p-1}} \right| \to 0  \;\; \text{as}\;\; t \to T,
\end{equation}
where $\xi^\alpha = \prod \limits_{i=1}^N \xi_i^{\alpha_i}$, $|\alpha| = \sum\limits_{i=1}^N \alpha_i$ if $\alpha = (\alpha_1, \dots, \alpha_n) \in \mathbb{N}^N$ and $\sum \limits_{|\alpha| = m}c_\alpha\xi^\alpha \geq 0$ for all $\xi \in \RN$.\\

According to Vel\'azquez \cite{VELcpde92}, if case 1 occurs with $\ell_{\hat{a}} = N$ or case 2 occurs with $\sum_{|\alpha| = m}c_\alpha\xi^\alpha > 0$ for all $\xi \ne 0$, then $\hat{a}$ is an isolated blow-up point. Herrero and Vel\'azquez \cite{HVasnsp92} and \cite{HVcrasp92} proved that the profile \eqref{def:fl} with $\ell_{\hat{a}} = N$ is generic in the case $N = 1$, and they announced the same for $N \geq 2$, but they never published it. Bricmont and Kupiainen \cite{BKnon94}, Merle and Zaag \cite{MZdm97} show the existence of initial data for \eqref{equ:problem} such that the corresponding solutions blow up in finite time $T$ at only one blow-up point $\hat{a}$ and verify the behavior \eqref{equ:case1} with $\ell_{\hat{a}} = N$. The method of \cite{MZdm97} also gives the stability of the profile \eqref{def:fl} ($\ell_{\hat{a}} = N$) with respect to perturbations in the initial data (see also Fermanian, Merle and Zaag \cite{FMZma00}, \cite{FZnon00} for other proofs of the stability). In \cite{EZsema11} and \cite{NZsns16}, the authors prove the stability of the profile \eqref{def:fl} ($\ell_{\hat{a}} = N$) with respect to perturbations in the initial data and also in the nonlinearity, in some class allowing lower order terms in the solution and also in the gradient. All the other asymptotic behaviors are suspected to be unstable. 

When $$\ell_{\hat{a}} \leq N - 1$$
in \eqref{equ:case1}, we do not know whether $\hat{a}$ is isolated or not, or whether $S$ is continuous near $\hat{a}$. In this paper, we assume that $\hat{a}$ is a non-isolated blow-up point and that $S$ is continuous locally near $\hat{a}$, in a sense that we will precisely define later. Our main concern is the regularity of $S$ near $\hat{a}$. The first relevant result is due to Vel\'azquez \cite{VELiumj93} where the author showed that the Hausdorff measure of $S$ is less or equal to $N - 1$. No further results on the description of $S$ were known until the contributions of Zaag \cite{ZAAaihp02}, \cite{ZAAcmp02} and \cite{ZAAdm06} (see also \cite{ZAAmme02} for a summarized  note). In \cite{ZAAaihp02}, the author proves that if $S$ is locally continuous, then $S$ is a $\mathcal{C}^1$ manifold. He also obtains the first singularity description near $\hat{a}$. More precisely, he shows that (see Theorems 3 and 4 in \cite{ZAAaihp02}) for some $t_0 < T$ and $\delta > 0$, for all $K_0 > 0$, $t \in [t_0, T)$ and $x \in B(\hat{a},2\delta)$ such that $d(x,S) \leq K_0\sqrt{(T-t)|\log(T-t)|}$, 
\begin{equation}\label{equ:asynerabar}
\left|(T-t)^\frac{1}{p-1}u(x,t) - f_1\left(\frac{d(x,S)}{\sqrt{(T-t)|\log (T-t)|}}\right) \right|\leq C(K_0)\frac{\log |\log (T-t)|}{|\log(T-t)|},
\end{equation}
where $f_1$ is defined in \eqref{def:fl} ($\ell_{\hat{a}} = 1$). Moreover, for all $x \in \RN \setminus S$, $u(x,t) \to u^*(x)$ as $t \to T$ with 
\begin{equation}\label{def:ustar}
u^*(x) \sim U(d(x,S)) = \left(\frac{8p}{(p-1)^2} \frac{|\log d(x,S)|}{d^2(x,S)}\right)^\frac{1}{p-1} \quad \text{as} \;\; d(x,S) \to 0 \;\; \text{and} \;\; x \in B(\hat{a},2\delta).
\end{equation}
If 
$$\ell_{\hat{a}} = 1,$$
Zaag in \cite{ZAAcmp02} further refines the asymptotic behavior \eqref{equ:asynerabar} and gets to error terms of order $(T - t)^\mu$ for some $\mu > 0$. This way, he obtains more regularity on the blow-up set $S$. The key idea is to replace the explicit profile $f_1$ in \eqref{equ:asynerabar} by a non-explicit function, say $\tilde{u}(x_1,t)$, then to go beyond all logarithmic scales through scaling and matching. In fact, for $\tilde{u}(x_1,t)$, Zaag takes a symmetric, one dimensional solution of \eqref{equ:problem} that blows up at the same time $T$ only at the origin, and behaves like \eqref{equ:case1} with $\ell_{\hat{a}} = 1$. More precisely, he abandons the explicit profile function $f_1$ in \eqref{equ:asynerabar} and chooses a non explicit function $\tilde{u}_{\sigma}(d(x,S),t)$ as a first order description of the singular behavior, where $\tilde{u}_\sigma$ is defined by 
\begin{equation}\label{def:utilde}
\tilde{u}_\sigma(x_1,t) = e^{-\frac{\sigma}{p-1}}\tilde{u}\left(e^{-\frac{\sigma}{2}}x_1, T - e^{-\sigma}(T-t)\right).
\end{equation}
He shows that for each blow-up point $a$ near $\hat{a}$, there is an optimal scaling parameter $\sigma = \sigma(a)$ so that the difference $(T-t)^\frac{1}{p-1}\left(u(x,t) - \tilde{u}_{\sigma(a)}(d(x,S),t)\right)$ along the normal direction to $S$ at $a$ is minimum. Hence, if the function $\tilde{u}_{\sigma(a)}(d(x,S),t)$ is chosen as a first order description for $u(x,t)$ near $(a,T)$, we escape logarithmic scales. More precisely, for all $t \in [t_0, T)$ and $x \in B(\hat{a},2\delta)$ such that $d(x,S) \leq K_0\sqrt{(T-t)|\log(T-t)|}$,
\begin{equation}\label{est:Zaacmp02}
(T-t)^\frac{1}{p-1}\left|u(x,t) - \tilde{u}_{\sigma(a)}(d(x,S),t)\right| \leq C(T-t)^\mu,
\end{equation}
for some $\mu > 0$. Note that any other value of $\sigma \ne \sigma(a)$ in \eqref{est:Zaacmp02} gives an error of logarithmic order of the variable $(T-t)$ (the same as in \eqref{equ:asynerabar}). Exploiting estimate \eqref{est:Zaacmp02} yields geometric constraints on $S$ which imply the $\mathcal{C}^{1, \frac{1}{2} - \eta}$-regularity of $S$ for all $\eta > 0$. A further refinement of \eqref{est:Zaacmp02} given in \cite{ZAAdm06} yields better estimates in the expansion of $u(x,t)$ near $(a,T)$. Moreover, some following terms in the expansion of $u(x,t)$ near $(a,T)$ contain geometrical descriptions of $S$, resulting in more regularity of $S$, namely the $\mathcal{C}^2$-regularity. \\

In this work, we want to know whether the $\mathcal{C}^2$-regularity near $\hat{a}$ proven in \cite{ZAAdm06} for $\ell_{\hat{a}} = 1$ would hold in the case where $u$ behaves like \eqref{equ:case1} near $(\hat{a},T)$ with 
\begin{equation}\label{equ:ellain2N1}
\ell_{\hat{a}} \in \{2, \dots, N-1\}.
\end{equation}
Since the author in \cite{ZAAaihp02} and \cite{ZAAdm06} obtains the result only when $\ell_{\hat{a}} = 1$, this corresponds to $(N-1)$-dimensional blow-up set (the codimension of the blow-up set is one, according to \cite{ZAAaihp02}). In our opinion, in those papers the major obstacle towards the case \eqref{equ:ellain2N1} lays in the fact that the author could not refine the asymptotic behavior \eqref{equ:case1} with $\ell_{\hat{a}} \in \{2, \dots, N-1\}$ to go beyond all logarithmic scales and get a smaller error term in polynomial orders of the variable $(T-t)$. It happens that a similar difficulty was already encountered by Fermanian and Zaag in \cite{FZnon00}, when they wanted to find a sharp
profile in the case \eqref{equ:case1} with $\ell_{\hat{a}} = N$, which corresponds to an isolated blow-up point, as we have pointed  out right after estimate \eqref{equ:case2}. Such a sharp profile could be obtained in \cite{FZnon00} only when $N=1$ (which corresponds also to $\ell_{\hat{a}}=1$): no surprise it was $\tilde u_\sigma(x_1,t)$, the dilated version of $\tilde u(x_1,t)$, the one-dimensional blow-up solution mentioned between estimates \eqref{def:ustar} and \eqref{def:utilde}. As a matter of fact, the use of $\tilde u(x_1,t)$ was first used in \cite{FZnon00} for the isolated blow-up point in one space dimension ($N=1$ and $\ell_{\hat{a}}=1$), then later in higher dimensions with a $(N-1)$ dimensional blow-up surface ($N\ge 2$ and still $\ell_{\hat{a}}=1$) in \cite{ZAAcmp02}.

The interest of $\tilde u(x_1,t)$ is that it provides a one-parameter family of blow-up solutions, thanks to the scaling parameter in \eqref{def:utilde}, which enables to get the sharp profile by suitably choosing the parameter.

Handling the case $\ell_{\hat a}\ge 2$ remained open, both for the case of an isolated point ($\ell_{\hat a}=N\ge 2$) and a non-isolated blow-up point ($\ell_{\hat a}=2,\cdots,N-1$). From the refinement of the expansion around the explicit profile in $f_{\ell_{\hat a}}$ in \eqref{equ:case1}, it appeared than one needs a $\frac{\ell_{\hat a}(\ell_{\hat a}+1)}{2}$-parameter family of blow-up solutions obeying \eqref{equ:case1}.

Such a family was constructed by Nguyen and Zaag in \cite{NZens16}, and successfully used to derive a sharp profile in the case of an isolated blow-up point ($\ell_{\hat a}=N \ge 2$), by fine-tuning the $\frac{\ell_{\hat a}(\ell_{\hat a}+1)}{2}= \frac{N(N+1)}{2}$ parameters. 

In this paper, we aim at using that family to handle the case of a non-isolated blow-up point ($N\ge 2$ and $\ell_{\hat a} =2,\cdots,N-1$), in order to generalize the results of Zaag in \cite{ZAAaihp02}, \cite{ZAAcmp02} and \cite{ZAAdm06}, proving in particular the $C^2$ regularity of the blow-up set, under the mere hypothesis that it is continuous.

The main result in this paper is the following.
\begin{theorem}[$\mathcal{C}^2$ regularity of the blow-up set assuming $\mathcal{C}^1$ regularity]\label{theo:1} Take $N \geq 2$ and $\ell \in \{1, \cdots, N-1\}$. Consider $u$ a solution of \eqref{equ:problem} that blows up in finite time $T$ on a set $S$ and take $\hat{a} \in S$ where $u$ behaves locally as stated in \eqref{equ:case1} with $\ell_{\hat{a}} = \ell$. If $S$ is locally a $\mathcal{C}^1$ manifold of dimension $N -\ell$, then it is locally $\mathcal{C}^2$.
\end{theorem}
\begin{remark} Theorem \ref{theo:1} was already proved by Zaag \cite{ZAAdm06} only when $\ell =1$. Thus, the novelty of our contribution lays in the case $\ell \in \{2, \dots, N-1\}$ and $N \geq 3$. 
\end{remark}
Under the hypotheses of Theorem \ref{theo:1}, Zaag \cite{ZAAaihp02} already proved that $S$ is a $\mathcal{C}^1$ manifold near $\hat{a}$, assuming that $S$ is continuous. Therefore, Theorem \ref{theo:1} can be restated under a weaker assumption. Before stating this stronger version, let us first clearly describe our hypotheses and introduce some terminologies borrowed from \cite{ZAAaihp02} (see also \cite{ZAAcmp02} and \cite{ZAAdm06}). According to Vel\'azquez \cite{VELcpde92} (see Theorem 2, page 1571), we know that for all $\epsilon >0$, there is $\delta(\epsilon) > 0$ such that 

$$S \cap B(\hat{a}, 2\delta) \subset \Omega_{\hat{a}, \epsilon} \equiv \left\{x \in \RN, \left|P_{\hat{a}}(x-\hat{a})\right| \geq (1 - \epsilon) |x - \hat{a}| \right\},
$$
where $P_{\hat{a}}$ is the orthogonal projection over $\pi_{\hat{a}}$, where 
$$\pi_{\hat{a}} = \hat{a} + \text{span}\{Q_{\hat{a}}^Te_{\ell_{\hat{a}} + 1}, \cdots, Q_{\hat{a}}^Te_N\}$$
is the so-called "weak" tangent plane to $S$ at $\hat{a}$. Roughly speaking, $\Omega_{\hat{a}, \epsilon}$ is a cone with vertex $\hat{a}$ and shrinks to $\pi_{\hat{a}}$ as $\epsilon \to 0$. In some "weak" sense, $S$ is $(N - \ell_{\hat{a}})$-dimensional. In fact, here comes our second hypothesis: we assume there is $\Gamma \in \mathcal{C}((-1,1)^{N-\ell_{\hat{a}}}, \RN)$ such that $\Gamma(0) = \hat{a}$ and $Im\, \Gamma \subset S$, where $Im\, \Gamma$ is at least $(N-\ell_{\hat{a}})$-dimensional, in the sense that 
\begin{equation}\label{hyp:coneprop}
\begin{array}{lll}
\quad & \text{$\forall b \in Im\,\Gamma$, there are $(N-\ell_{\hat{a}})$ independent vectors $v_1, \dots, v_{N-\ell_{\hat{a}}}$ in $\RN$ and}& \quad \\
\quad & \text{$\Gamma_1, \dots, \Gamma_{N-\ell_{\hat{a}}}$ functions in $\mathcal{C}^1([0,1],S)$ such that $\Gamma_i(0) = b$ and $\Gamma'_i(0) = v_i$.}&\quad
\end{array}
\end{equation}
The hypothesis \eqref{hyp:coneprop} means that $b$ is actually non-isolated in $(N-\ell_{\hat{a}})$ independent directions. We assume in addition that $\hat{a}$ is not an endpoint in $Im \,\Gamma$ in the sense that 
\begin{equation}\label{hyp:endpoint}
\begin{array}{lll}
\quad & \text{$\forall \epsilon > 0$, the projection of $\Gamma((-\epsilon, \epsilon)^{N-\ell_{\hat{a}}})$ on the "weak" tangent plane $\pi_{\hat{a}}$} &\quad \\
\quad & \text{at $\hat{a}$ contains an open ball centered at $\hat{a}$.}&\quad\\
\end{array}
\end{equation}

This is the stronger version of our result:

\noindent\textbf{Theorem \ref{theo:1}'.} \textit{Take $N \geq 2$ and $\ell \in \{1, \cdots, N-1\}$. Consider $u$ a solution of \eqref{equ:problem} that blows up in finite time $T$ on a set $S$ and take $\hat{a} \in S$ where $u$ behaves locally as stated in \eqref{equ:case1} with $\ell_{\hat{a}} = \ell$. Consider $\Gamma \in \mathcal{C}((-1,1)^{N - \ell}, \RN)$ such that $\hat{a} = \Gamma(0) \in Im\, \Gamma \subset S$ and $Im\, \Gamma$ is at least $(N - \ell)$-dimensional (in the sense \eqref{hyp:coneprop}). If $\hat{a}$ is not an endpoint (in the sense \eqref{hyp:endpoint}), then there are $\delta > 0$, $\delta_1 > 0$ and $\gamma \in \mathcal{C}^{2}((-\delta_1, \delta_1)^{N-\ell}, \mathbb{R}^{\ell})$ such that 
$$S_\delta = S \cap B(\hat{a},2\delta) = graph(\gamma) \cap B(\hat{a}, 2\delta) = Im\, \Gamma \cap B(\hat{a}, 2\delta),$$
and the blow-up set $S$ is a $\mathcal{C}^2$-hypersurface locally near $\hat{a}$.}\\

Let us now briefly give the main ideas of the proof of Theorem \ref{theo:1}. The proof is based on techniques developed by Zaag in \cite{ZAAcmp02} and \cite{ZAAdm06} for the case when the solution of equation \eqref{equ:problem} behaves like \eqref{equ:case1} with $\ell = 1$. As in \cite{ZAAcmp02} and \cite{ZAAdm06}, the proof relies on two arguments:
\begin{itemize}
\item[-] The derivation of a sharp blow-up profile of $u(x,t)$ near the singularity, in the sense that the difference between the solution $u(x,t)$ and this sharp profile goes beyond all logarithmic scales of the variables $(T-t)$. This is possible thanks to the recent result in \cite{NZens16}.
\item[-] The derivation of a refined asymptotic profile of $u(x,t)$ near the singularity linked to geometric constraints on the blow-up set. In fact, we derive an asymptotic profile for $u(x,t)$ in every ball $B(a, K_0\sqrt{T-t})$ for some $K_0 > 0$ and $a$ a blow-up point close to $\hat{a}$. Moreover, this profile is continuous in $a$ and the speed of convergence of $u$ to each one in the ball $B(a, K_0\sqrt{T-t})$ is uniform with respect to $a$. If $a$ and $b$ are in $S$ and $0 < |a - b| \leq K_0\sqrt{T-t}$, then the balls $B(a, K_0\sqrt{T-t})$ and $B(b,K_0\sqrt{T-t})$ intersect each other, leading to different profiles for $u(x,t)$ in the intersection. However, these profiles have to coincide, up to the error terms. This makes a geometric constraint which gives more regularity for the blow-up set near $\hat{a}$.
\end{itemize}
Let us explain the difficulty raised in \cite{ZAAcmp02} and \cite{ZAAdm06} for the case $\ell \geq 2$. Consider $a \in S \cap B(\hat{a},2\delta)$ for some $\delta > 0$ and introduce the following self-similar variables:
\begin{equation}\label{def:WaI}
W_a(y,s) = (T-t)^\frac{1}{p-1}u(x,t), \quad y = \frac{x - a}{\sqrt{T-t}}, \quad s = -\log(T-t).
\end{equation}
Then, we see from \eqref{equ:problem} that for all $(y,s) \in \RN \times [-\log T, +\infty)$,
\begin{equation}\label{equ:WaI}
\frac{\partial W_a}{\partial s} = \Delta W_a - \frac{1}{2}y\cdot \nabla W_a - \frac{W_a}{p-1}  + |W_a|^{p-1}W_a.
\end{equation}
Under the hypotheses stated in Theorem \ref{theo:1}, Zaag \cite{ZAAaihp02} proved in Proposition 3.1, page 513 and in Section 6.1, pages 530-533 that for all $a \in S_\delta \equiv S \cap B(\hat a, 2\delta)$ for some $\delta >0$ and $s \geq -\log T$, there exists an $(N \times N)$ orthogonal matrix $Q_a$ such that
\begin{equation}\label{est:WaI}
\left\|W_a(Q_ay,s) - \left\{\kappa + \frac{\kappa}{2ps} \left(\ell - \frac{|\bar{y}|^2}{2}\right) \right\}\right\|_{L^2_\rho} \leq C\frac{\log s}{s^2},
\end{equation}
where $\kappa = (p-1)^{-\frac{1}{p-1}}$, $\bar{y} = (y_1, \cdots, y_{\ell_a})$, $Q_a$ is continuous in terms of $a$ such that $\{Q_a^Te_j \vert j = \ell + 1, \dots, N\}$ spans the tangent plane $\pi_a$ to $S$ at $a$ and $Q_a^Te_i, i = 1, \cdots, \ell$ are the normal directions to $S$ at $a$, $L^2_\rho$ is the weighted $L^2$ space associated with the weight $\rho = \dfrac{1}{(4\pi)^{N/2}}e^{-\frac{|y|^2}{4}}$. Note that the estimate \eqref{est:WaI} implies \eqref{equ:asynerabar} (see Appendix C in \cite{ZAAaihp02}).

When $\ell =1$, in order to refine estimate \eqref{est:WaI}, the author in \cite{ZAAcmp02} subtracts from $W_a$ a 1-dimensional solution with the same profile. Let us do the same when $\ell = 2, \cdots, N-1$, and explain how the author succeeds in handing the case $\ell = 1$ and gets stuck when $\ell \geq 2$. To this end, we consider $\hat{u}(\bar{x}, t)$ with $\bar{x} = (x_1, \cdots, x_{\ell})$ a radially symmetric solution of \eqref{equ:problem} in $\mathbb{R}^{\ell}$ which blows up at time $T$ only at the origin with the profile \eqref{equ:case1} with $\ell_{\hat{a}} = \ell$ (see Appendix A.1 in \cite{NZens16} for the existence of such a solution). If the $\ell$-dimensional solution $\hat{u}$ is considered in $\RN$, then it blows up on the $(N-\ell)$ vector space $\{\bar{x} = 0\}$ in $\RN$. In particular, if we introduce 
\begin{equation}\label{def:whatI}
\hat{w}(\bar{y}, s) = (T-t)^\frac{1}{p-1}\hat{u}(\bar{x}, t), \quad \bar{y} = \frac{\bar{x}}{\sqrt{T-t}},\quad s = -\log(T-t),
\end{equation}
then, $\hat{w}$ is a radially symmetric solution of \eqref{equ:WaI} which satisfies
\begin{equation}\label{est:whatI}
\left\|\hat{w}(\bar{y},s) - \left\{\kappa + \frac{\kappa}{2ps} \left(\ell - \frac{|\bar{y}|^2}{2}\right) \right\}\right\|_{L^2_\rho} \leq C\frac{\log s}{s^2}.
\end{equation}
Noting that $\hat{u}$ and $\hat{w}$ maybe considered as solutions defined for all $y \in \RN$ (and independent of $y_{\ell+1}, \cdots, y_N$), and given that $\hat{w}(\bar{y},s)$ and $W_a(Q_ay,s)$ have the same behavior up to the first order (see \eqref{est:WaI} and \eqref{est:whatI}), we may try to use $\hat{w}$ as a sharper (though non-explicit) profile for $W_a(Q_ay,s)$. In fact, we have the following classification (see Corollary \ref{coro:class} below):\\

\noindent \textit{- Case 1: There is a symmetric, real $\ell_a \times \ell_a$ matrix $\mathcal{B} = \mathcal{B}(a) \ne 0$ such that 
\begin{equation}\label{equ:cas1clayb}
W_a(Q_ay,s) - \hat{w}(\bar{y},s) = \frac{1}{s^2}\left(\frac{1}{2}\bar{y}^T\mathcal{B}\bar{y} - tr(\mathcal{B})\right) + o\left(\frac{1}{s^2}\right)\;\; \text{as}\;\; s \to +\infty \;\; \text{in} \; L^2_\rho.
\end{equation}
- Case 2: There is a positive constant $C_0$ such that
\begin{equation}\label{equ:cas2clayb}
\|W_a(Q_ay,s) - \hat{w}(\bar{y},s) \|_{L^2_\rho} = \mathcal{O}\left(e^{-\frac s2}s^{C_0}\right) \;\; \text{as}\;\; s \to +\infty.
\end{equation}
}

\noindent If $\ell = 1$ ($\mathcal{B}(a) \in \mathbb{R}$), the author in \cite{ZAAcmp02} noted the following property
\begin{equation} \label{equ:casMEyb}
\hat{w}(y_1, s + \sigma_0) - \hat{w}(y_1,s) = \frac{2 \kappa \sigma_0}{ps^2}\left( \frac{1}{2}y_1^2 - 1\right) + o\left(\frac{1}{s^2}\right) \quad \text{in} \; L^2_\rho.
\end{equation}
Therefore, choosing $\sigma_0(a)$ such that $ \frac{2 \kappa \sigma_0}{p} = \mathcal{B}(a)$, we see from \eqref{equ:cas1clayb} and \eqref{equ:casMEyb} that 
$$W_a(Q_ay,s) - \hat{w}(y_1,s + \sigma_0(a)) = o\left(\frac{1}{s^2}\right)\;\; \text{as} \; s \to +\infty \quad \text{in} \; L^2_\rho.$$
From the classification given in \eqref{equ:cas1clayb} and \eqref{equ:cas2clayb}, only \eqref{equ:cas2clayb} holds and 
\begin{equation}\label{est:expIntro}
\|W_a(Q_ay,s) - \hat{w}(y_1,s + \sigma_0(a)) \|_{L^2_\rho} = \mathcal{O}\left(e^{-\frac s2}s^{C_0}\right)\;\; \text{as}\;\; s \to +\infty.
\end{equation}
If we return to the original variables $u(x,t)$ and $\hat{u}(x_1,t)$ through \eqref{def:WaI} and \eqref{def:whatI}, then \eqref{est:Zaacmp02} follows from  the transformation \eqref{def:utilde} together with estimate \eqref{est:expIntro} (see Appendix C in \cite{ZAAcmp02}). In other words, $\hat{w}(y_1, s+\sigma_0(a))$ serves as a sharp (though non-explicit) profile for $W_a(Q_ay,s)$ in the sense of \eqref{est:expIntro}. Using estimate \eqref{est:expIntro} together with some geometrical arguments, we are able to prove the $\mathcal{C}^{1, \frac 12 - \eta}$-regularity of the blow-up set, for any $\eta > 0$. Then, a further refinement of \eqref{est:expIntro} up to order of $\frac{e^{-\frac s2}}{s}$ together with a geometrical constraint on the blow-up set $S$ results in more regularity for $S$, which yields the $\mathcal{C}^2$-regularity.\\

If $\ell \geq 2$, the matrix $\mathcal{B}(a)$ in \eqref{equ:cas1clayb} has $\frac{\ell(\ell + 1)}{2}$ real parameters. Therefore, applying the trick of \cite{ZAAcmp02} (see \eqref{equ:casMEyb} above) only allows to manage one parameter; there remain $\frac{\ell(\ell + 1)}{2} - 1$ real parameters to be handled. This is the major reason which prevents the author in \cite{ZAAcmp02} and \cite{ZAAdm06} from deriving a similar estimate to \eqref{est:expIntro}, hence, the refined regularity of the blow-up set. Fortunately, we could overcome this obstacle thanks to a recent result by Nguyen and Zaag \cite{NZens16} (see Proposition \ref{prop:cons} below) where the authors show that for all symmetric, real $\ell \times \ell$ matrix $\mathcal{A}$, there is a solution $w_{\mathcal{A}}$ of equation \eqref{equ:WaI} in $\mathbb{R}^\ell$ such that 
\begin{equation}\label{equ:NZ15cons}
 w_{\mathcal{A}}(\bar{y},s) - \hat{w}(\bar{y},s) = \frac{1}{s^2}\left(\frac{1}{2}\bar{y}^T\mathcal{A}\bar{y} - tr(\mathcal{A})\right) + o\left(\frac{1}{s^2}\right) \quad \text{as} \; s \to +\infty\;\; \text{in} \; L^2_\rho.
\end{equation}
Hence, choosing $\mathcal{A} = \mathcal{B}(a)$, we see from \eqref{equ:NZ15cons}, \eqref{equ:cas1clayb} and \eqref{equ:cas2clayb} that 
\begin{equation}\label{est:expIntro1}
\|W_a(Q_ay,s) - w_{\mathcal{B}(a)}(\bar{y},s)\|_{L^2_\rho} \leq Ce^{-\frac s2}s^{C_0}, \quad \text{for $s$ large enough}.
\end{equation}
Exploiting estimate \eqref{est:expIntro1} and adapting the arguments given in \cite{ZAAcmp02} and \cite{ZAAdm06}, we are able to prove the $\mathcal{C}^2$-regularity of the blow-up set. \\

The next result shows how the $\mathcal{C}^2$-regularity is linked to the refined asymptotic behavior of $W_a$. More precisely, we link in the following theorem the refinement of the asymptotic behavior of $W_a$ to the second fundamental form of the blow-up set at $a$.
\begin{theorem}[Refined asymptotic behaviors linked to the geometrical description of the blow-up set]\label{theo:2} Under the hypotheses of Theorem \ref{theo:1}, there exists $\tilde{s}_0 \geq -\log T$ and $\delta > 0$ such that for all $a \in S_\delta = S \cap B(\hat a, 2\delta)$, there exists a continuous $(\ell \times \ell)$ symmetric matrix $\mathcal{B}(a)$ such that for all $s \geq \tilde{s}_0$, 
\begin{align}
\left\|W_a(Q_ay,s) - w_{\mathcal{B}(a)}(\bar{y},s) - \frac{\kappa e^{-\frac{s}{2}}}{2ps}\sum_{i = 1}^{\ell}y_i\sum_{k,j = \ell + 1}^N \frac{\Lambda^{(i)}_{k,j}(a)}{1 + \delta_{k,j}} (y_ky_j - 2\delta_{k,j})\right\|_{L^2_\rho}\leq C\frac{e^{-\frac{s}{2}}}{s^{\frac 32 - \nu}},\label{est:link}
\end{align}
for some $\nu \in (0, \frac 12)$, where $a \to \{\Lambda^{(i)}_{k,j}(a)\}_{\ell+1\leq j,k\leq N}$ is a continuous symmetric matrix representing the second fundamental form of the blow-up set at the blow-up point $a$ along the unitary normal vector $Q_a^Te_i$. Moreover, 
\begin{equation}\label{est:limitLam}
\Lambda^{(i)}_{k,j}(a) = \frac{p}{4\kappa}\lim_{s \to +\infty}se^{\frac s2}\int_{\RN}W_a(Q_ay,s)y_i(y_ky_j - 2\delta_{k,j})\rho(y)dy.
\end{equation}
\end{theorem}

\bigskip

In Section \ref{sec:2}, we give the main steps of the proof of Theorems \ref{theo:1} and \ref{theo:2}. We leave all long and technical proofs to Section \ref{sec:proofAll}.

\section{Setting of the problem and strategy of the proof of the $\mathcal{C}^2$-regularity of the blow-up set.} \label{sec:2}
In this section we give the main steps of the proof of Theorems \ref{theo:1} and \ref{theo:2}. All long and technical proofs will be left to the next section. We proceed in 3 parts corresponding to 3 separate subsections. For the reader's convenience, we briefly describe these parts as follows:
\begin{itemize}
\item Part 1: We derive a sharp blow-up behavior for solutions of equation \eqref{equ:problem} having the profile \eqref{equ:case1} with $\ell_{\hat{a}} \in \{1, \cdots, N-1\}$ such that the difference between the solution and this sharp blow-up behavior goes beyond all logarithmic scales of the variable $T-t$. The main result in this step is stated in Proposition \ref{prop:expodecay}. 
\item Part 2: Through the introduction of a local chart, we give a geometrical constraint on the expansion of the solution linked to the asymptotic behavior (see Proposition \ref{prop:geocon} below). This geometrical constraint is a crucial point which is the bridge between the asymptotic behavior and the regularity of the blow-up set. 
\item Part 3: Using the sharp blow-up behavior derived in Part 1, we first get the $\mathcal{C}^{1, \frac{1}{2} - \eta}$- regularity of the blow-up set $S$ (see Proposition \ref{prop:C1alph} below), then together with the geometrical constraint, we achieve the $\mathcal{C}^{1, 1 - \eta}$-regularity of $S$ (see Proposition \ref{prop:C11eta} below). With this better regularity and the geometric constraint, we further refine the asymptotic behavior (see Proposition \ref{prop:furRe} below) and use again the geometric constraint to get $\mathcal{C}^2$ -regularity of $S$, which yields the conclusion of Theorems \ref{theo:1} and \ref{theo:2}. 
\end{itemize}

The reader should be noticed that Parts 1 and 2 are independent, whereas Part 3 is a combination of the first two parts. Throughout this paper, we work under the hypotheses of Theorem \ref{theo:1}. Since $S$ is locally near $\hat a$ a manifold of dimension $N - \ell$, we may assume that there is a $\mathcal{C}^1$ function $\gamma$ such that 
\begin{equation}\label{def:Sdelta}
S_\delta \equiv S \cap B(\hat{a}, 2\delta) = \text{graph}(\gamma) \cap B(\hat{a}, 2\delta),
\end{equation}
for some $\delta >0$ and $\gamma \in \mathcal{C}^1((-\delta_1,\delta_1)^{N - \ell}, \mathbb{R}^\ell)$ with $\delta_1 > 0$.

In what follows, $\ell \in \{1, \cdots, N-1\}$ is fixed, and for all $z = (z_1, \cdots, z_N) \in \RN$, we denote by $\bar{z}$ the first $\ell$ coordinates of $z$, namely $\bar{z} = (z_1, \cdots, z_\ell)$, and by $\tilde{z}$ the last $(N-\ell)$ coordinates of $z$, namely $\tilde{z} = (z_{\ell + 1}, \cdots, z_N)$. We usually use indices $i$, $m$ for the range $1, \cdots, \ell$ and indices $j$, $k$, $n$ for the range $\ell + 1, \cdots, N$.

\subsection{Part 1: Blow-up behavior beyond all logarithmic scales of $(T-t)$.}
In this subsection, we use the ideas given by Fermanian and Zaag \cite{FZnon00} together with a recent result by Nguyen and Zaag in \cite{NZens16} in order to derive a sharp (though non-explicit) profile for blow-up solutions of \eqref{equ:problem} in the sense that the first order in the expansion of the solution around this sharp profile goes beyond all logarithmic scales of $(T-t)$ and reaches to polynomial scales of $(T-t)$. In fact, we replace the 1-scaling parameter $\sigma$ in \eqref{est:Zaacmp02} by a $\frac{\ell(\ell+1)}{2}$-parameters family, which generates a substitution for $\tilde{u}_\sigma$ \eqref{def:utilde} and serves as a sharp profile for solutions having the behavior \eqref{equ:case1} with $\ell_{\hat{a}} \in \{1, \dots, N-1\}$. The main result in this part is Proposition \ref{prop:expodecay} below.

Consider $a \in S_\delta$. If $W_a(y,s)$ and $\hat{w}(\bar{y},s)$ are defined as in \eqref{def:WaI} and \eqref{def:whatI}, then we know from \cite{ZAAaihp02} that
\begin{equation}\label{est:Wa}
\left\|W_a(Q_ay,s) - \left\{\kappa + \frac{\kappa}{2ps} \left(\ell - \frac{|\bar{y}|^2}{2}\right) \right\}\right\|_{L^2_\rho} \leq C\frac{\log s}{s^2},
\end{equation}
and 
\begin{equation}\label{est:what}
\left\|\hat{w}(\bar{y},s) - \left\{\kappa + \frac{\kappa}{2ps} \left(\ell - \frac{|\bar{y}|^2}{2}\right) \right\}\right\|_{L^2_\rho} \leq C\frac{\log s}{s^2}.
\end{equation}

The first step is to classify all possible asymptotic behaviors of $W_a(Q_ay, s) - \hat w(\bar y, s)$ as $s$ goes to infinity. To do so, we shall use the following result which is inspired by Fermanian and Zaag \cite{FZnon00}: 

\begin{proposition}[Classification of the difference between two solutions of \eqref{equ:WaI} having the same profile]\label{prop:class} Assume that $W_1$ and $W_2$ are two solutions of \eqref{equ:WaI} verifying
\begin{equation}\label{hyp:W1W2}
i = 1,2, \quad \left\|W_i(y,s) - \left\{\kappa + \frac{\kappa}{2ps} \left(\ell - \frac{|\bar y|^2}{2}\right) \right\}\right\|_{L^2_\rho} \leq C\frac{\log s}{s^2},
\end{equation}
where $\bar{y} = (y_1, \cdots, y_\ell)$ for some $\ell \in \{1, \cdots, N - 1\}$. Then, one of the two following cases occurs: \\
- Case 1: there is a symmetric, real $(\ell \times \ell)$ matrix $\mathcal{B} \ne 0$ such that 
\begin{equation}\label{equ:case1class}
W_1(y,s) - W_2(y,s) = \frac{1}{s^2}\left(\frac{1}{2}\bar{y}^T \mathcal{B}\bar{y} - tr(\mathcal{B})\right) + o\left(\frac{1}{s^2}\right) \quad \text{as} \;\; s \to +\infty \;\;\text{in}\; L^2_\rho.
\end{equation}
- Case 2: there is $C_0 > 0$ such that 
\begin{equation}\label{equ:case2class}
\|W_1(y,s) - W_2(y,s)\|_{L^2_\rho} = \mathcal{O}\left(e^{-\frac s2}s^{C_0}\right) \quad \text{as} \; s \to +\infty.
\end{equation}
\end{proposition}
\begin{proof} The proof follows from the strategy given in \cite{FZnon00} for the difference of two solutions with the radial profile $(\ell = N)$. Note that the case when $\ell = 1$ was treated in \cite{ZAAcmp02}. Since some technical details are straightforward, we briefly give the main steps of the proof in Section \ref{sec:3} and just emphasize on the novelties. 
\end{proof}

An application of Proposition \ref{prop:class} with $W_1(y,s) = W_a(Q_ay,s)$ and $W_2(y,s) = \hat w(\bar y, s)$ yields the following corollary directly:
\begin{corollary} \label{coro:class}  As $s$ goes to infinity, one of the two following cases occurs:\\
- Case 1: there is a symmetric, real $(\ell \times \ell)$ matrix $\mathcal{B} = \mathcal{B}(a) \ne 0$ continuous as a function of $a$ such that 
\begin{equation}\label{equ:case1class}
W_a(Q_ay,s) - \hat{w}(\bar{y},s) = \frac{1}{s^2}\left(\frac{1}{2}\bar{y}^T \mathcal{B}\bar{y} - tr(\mathcal{B})\right) + o\left(\frac{1}{s^2}\right) \quad \text{in}\; L^2_\rho.
\end{equation}
- Case 2: there is $C_0 > 0$ such that 
\begin{equation}\label{equ:case2class}
\|W_a(Q_ay,s) - \hat{w}(\bar{y},s)\|_{L^2_\rho} = \mathcal{O}\left(e^{-\frac s2}s^{C_0}\right).
\end{equation}
\end{corollary}
\begin{remark} Note that the continuity of $\mathcal{B}$ comes from the continuity of $W_a$ with respect to $a$, where $W_a$ behaves as in \eqref{est:Wa}. In particular, Zaag \cite{ZAAaihp02} showed the stability of the blow-up behavior \eqref{est:Wa} with respect to blow-up points (see Proposition 3.1 and Section 6.1 in \cite{ZAAaihp02}).
\end{remark}

In the next step, we recall the recent result by Nguyen and Zaag \cite{NZens16}, which gives the construction of solutions for equation \eqref{equ:WaI} with some prescribed behavior. 
\begin{proposition}[Construction of solutions for \eqref{equ:WaI} with some prescribed behavior]\label{prop:cons} Consider $\ell \in \{1, \cdots, N-1\}$. For all $\mathcal{A} \in \mathbb{M}_\ell(\R)$, where $\mathbb{M}_\ell(\R)$ is the set of all symmetric, real $(\ell \times \ell)$-matrix, there exists a solution $w_\mathcal{A}(y,s)$ of \eqref{equ:WaI} defined on $\mathbb{R}^N \times [s_0(\mathcal{A}), +\infty)$ such that
\begin{equation}\label{est:whatwA}
w_\mathcal{A}(\bar{y},s) - \hat{w}(\bar{y},s) = \frac{1}{s^2}\left(\frac{1}{2}\bar{y}^T\mathcal{A}\bar{y} - tr(\mathcal{A})\right) + o\left(\frac{1}{s^2}\right) \quad \text{as}\; s \to +\infty \;\; \text{in}\;\; L^2_\rho,
\end{equation}
where $\hat{w}$ is the radially symmetric, $\ell$-dimensional solution of \eqref{equ:WaI} satisfying \eqref{est:what}.
\end{proposition}
\begin{proof} See Theorem 3 in \cite{NZens16}. Although that result is stated for the case $\ell = N$, we can extend it to the case when $\ell \leq N-1$ by considering solutions of \eqref{equ:WaI} as $\ell$-dimensional solutions, those artificially generated by adding irrelevant space variables $(y_{\ell + 1}, \cdots, y_{N})$ to the domain of definition of the solutions. 
\end{proof}

The following result is a direct consequence of Corollary \ref{coro:class} and Proposition \ref{prop:cons}:
\begin{proposition}[Sharp (non-explicit) profile for solutions of \eqref{equ:problem} having the behavior \eqref{equ:case1} with $\ell \leq N-1$]\label{prop:expodecay} There exist $s_0 > 0$ and a continuous matrix $\mathcal{B}:\;S_\delta \to M_{\ell}(\R)$, such that for all $a \in S_\delta$ and $s \geq s_0$,
\begin{equation}\label{est:expo_s}
\left\|W_a(Q_ay,s) - w_{\mathcal{B}(a)}(\bar{y}, s)\right\|_{L^2_\rho} \leq Ce^{-\frac s2}s^{C_0},
\end{equation}
where $w_{\mathcal{B}}$ is the solution constructed as in Proposition \ref{prop:cons}, $C_0>0$ is given in Proposition \ref{prop:class}. Moreover, we have \\
$(i)\;\;$ For all $s \geq s_0 + 1$,
\begin{equation}\label{est:extendys}
\sup_{|y| \leq K\sqrt{s}}\left|W_a(y, s) - w_{\mathcal{B}(a)}(\bar{y}_a, s)\right| \leq C(K)e^{-\frac s2}s^{\frac{3}{2} + C_0},
\end{equation}
where $\bar{y}_a = (y\cdot Q_ae_1, \cdots, y\cdot Q_ae_\ell)$. \\
$(ii)\;$ For all $t \in [T - e^{-s_0-1}, T)$, 
\begin{align}
\sup_{|x - a| \leq K\sqrt{(T-t)|\log (T-t)|}}&\left|(T-t)^{\frac{1}{p-1}}u(x,t) - w_{\mathcal{B}(a)}(\bar{y}_{a,x}, -\log(T-t))\right|\nonumber\\
&\qquad \qquad \qquad\leq C(K)(T-t)^{\frac{1}{2}}|\log(T-t)|^{\frac{3}{2} + C_0},\label{est:beyLogT_t}
\end{align}
where $\bar{y}_{a,x} = \frac{1}{\sqrt{T-t}}((x-a)\cdot Q_ae_1,\cdots, (x-a)\cdot Q_ae_\ell)$.
\end{proposition} 
\begin{proof}  From \eqref{equ:case1class} and \eqref{est:whatwA}, we have for any $\ell \times \ell$ symmetric matrix $\mathcal{A}$,
$$W_{a}(Q_ay,s) - w_{\mathcal{A}}(\bar{y},s) = \frac{1}{s^2} \left(\frac{1}{2}\bar{y}^T(\mathcal{B} - \mathcal{A})\bar{y} - tr (\mathcal{B} - \mathcal{A}) \right) + o\left(\frac{1}{s^2}\right) \quad \text{in} \;\; L^2_\rho.$$
Choosing $\mathcal{A} = \mathcal{B}(a)$, we get
\begin{equation}\label{est:tmpWawA}
\left\|W_{a}(Q_ay,s) - w_{\mathcal{B}(a)}(\bar{y},s)\right\|_{L^2_\rho} = o\left(\frac{1}{s^2}\right), \quad \text{as} \; s \to+\infty.
\end{equation}
Note that an alternative application of Proposition \ref{prop:class} with $W_1 = W_a$ and $W_2 = w_{\mathcal{B}(a)}$ yields either \eqref{equ:case1class} or \eqref{equ:case2class}. However, the case \eqref{equ:case1class} is excluded by \eqref{est:tmpWawA}. Hence, \eqref{est:expo_s} follows. Since we showed in Corollary \ref{coro:class} that $a \mapsto \mathcal{B}(a)$ is continuous, the same holds for $a \mapsto \mathcal{A}(a)$. 

As for \eqref{est:extendys}, it is a direct consequence of the following lemma which allows us to carry estimate \eqref{est:expo_s} from compact sets $|y| \leq K$ to sets $|y| \leq K\sqrt{s}$:
\begin{lemma}[Extension of the convergence from compact sets to sets $|y| \leq K\sqrt{s}$] \label{lemm:extVel} Assume that $Z$ satisfies 
\begin{equation}\label{eq:Zvel}
\partial_s Z \leq \Delta Z - \frac{1}{2}y\cdot \nabla Z + Z + \frac{C_1}{s}Z,  \quad 0 \leq Z(y,s) \leq C_1, \quad \forall (y,s) \in \RN \times [\hat{s},+\infty),
\end{equation}
for some $C_1 > 0$. Then for all $s' \geq \hat{s}$ and $s \geq s' + 1$ such that $e^{\frac{s - s'}{2}} = \sqrt{s}$, we have
$$\sup_{|y| \leq K\sqrt{s}}Z(y,s) \leq C(C_1, K)e^{s-s'}\|Z(s')\|_{L^2_\rho}.$$
\end{lemma}
\begin{proof} This lemma is a corollary of Proposition 2.1 in Vel\'azquez \cite{VELcpde92} and it is proved in the course of the proof of Proposition 2.13 in \cite{FZnon00} (in particular, pp. 1203-1205).
\end{proof}

Let us derive \eqref{est:extendys} from Lemma \ref{lemm:extVel}. If we define $G(y,s) = W_a(Q_ay,s) - w_{\mathcal{B}(a)}(\bar{y},s)$, straightforward calculations based on \eqref{equ:WaI} yield 
\begin{equation}\label{eq:GWawB}
\partial_s G = \Delta G - \frac{1}{2}y\cdot \nabla G + G + \alpha G, \quad \forall (y,s) \in \mathbb{R}^N \times [-\log T,+\infty),
\end{equation}
where 
\begin{equation*}
\alpha(y,s) = \frac{|W_a|^{p-1}W_a - |w_{\mathcal{B}}|^{p-1} w_{\mathcal{B}}}{W_a - w_{\mathcal{B}}} - \frac{p}{p-1} = p|\tilde w(y,s)|^{p-1} - \frac{p}{p-1} \quad \text{if}\;\; W_a \ne w_{\mathcal{B}},
\end{equation*}
for some $\tilde w(y,s) \in \left(W_a(Q_ay,s), w_{\mathcal{B}(a)}(\bar y,s)\right)$. 

From Merle and Zaag \cite{MZgfa98} (Theorem 1), we know that for $s$ large enough,
$$ \|\tilde w(s)\|_{L^\infty} \leq \kappa + \frac{C}{s},$$
which follows
\begin{equation}\label{eq:alWawB}
\alpha(y,s) \leq p \left( \kappa + \frac{C}{s}\right)^{p-1} -\frac{p}{p-1} \leq \frac{C_1}{s}.
\end{equation}
If $Z = |G|$, then we use the Kato's inequality $\Delta G\cdot \text{sgn}(G) \leq \Delta (|G|)$ to derive equation \eqref{eq:Zvel} from \eqref{eq:GWawB} and \eqref{eq:alWawB}. Applying Lemma \ref{lemm:extVel} together with estimate \eqref{est:expo_s} yields for all $s' \geq s_1$ and $s \geq s' + 1$ for some $s_1 > 0$ large such that $e^{\frac{s - s'}{2}} = \sqrt{s}$,
$$\sup_{|y| \leq K\sqrt{s}}Z(y,s) \leq Ce^{s-s'}e^{-\frac{s'}{2}}(s')^{C_0} \leq C e^{-\frac{s}{2}}s^{\frac{3}{2} + C_0},$$
which yields \eqref{est:extendys}. The estimate \eqref{est:beyLogT_t} directly follows from \eqref{est:extendys} by the transformation \eqref{def:WaI}. This ends the proof of Proposition \ref{prop:expodecay}.
\end{proof}

\subsection{Part 2: A geometric constraint linked to the asymptotic behaviors.}
In this subsection, we follow the idea of \cite{ZAAdm06} to introduce local $\mathcal{C}^{1, \alpha^*}$-charts of the blow-up set, and get a geometric constraint mechanism on the blow-up set (see Proposition \ref{prop:geocon} below) which is a crucial step in linking refined asymptotic behaviors of the solution to geometric descriptions of the blow-up set. 

Consider $a \in S_\delta$ and $\ell \in \{1, \cdots, N-1\}$, we introduce the local $\mathcal{C}^{1, \alpha^*}$-chart of the blow-up set at the point $a$ as follows:
\begin{align*}
\mathbb{R}^{N - \ell}\;\; &\to \quad \quad\RN\\
\tilde{\xi}\;\;\;\;\; &\mapsto \;\;(\gamma_{a,1}(\tilde{\xi}), \cdots, \gamma_{a,\ell}(\tilde{\xi}), \tilde{\xi}),
\end{align*}
where $\tilde{\xi} = (\xi_{\ell+1}, \cdots, \xi_N)$ and  $\gamma_{a,i} \in \mathcal{C}^{1, \alpha^*}((-\epsilon_a, \epsilon_a)^{N - \ell})$ for some $\alpha^* \in \left(0,\frac 12\right)$ and $\epsilon_a > 0$, then the set $S_\delta$ is locally near $a$ defined  by
\begin{equation}\label{def:localChart}
\left\{a + \sum_{i = 1}^\ell \gamma_{a,i}(\tilde{\xi})\eta_i(a) + \sum_{j = \ell + 1}^N \xi_k\tau_k(a)\; \big \vert\; |\tilde{\xi}| < \epsilon_a\right\},
\end{equation}
where $\eta_1(a), \cdots, \eta_\ell(a)$ and $\tau_{\ell + 1}(a), \cdots, \tau_N(a)$ are of norm 1, and respectively, normal and tangent to $S_\delta$ at $a$. By definition, we have 
\begin{equation*}
\gamma_{a,i}(0) = 0 \quad \text{and} \quad \nabla\gamma_{a,i}(0) = 0, \;\forall i = 1,\cdots,\ell. 
\end{equation*}
Let $Q_a$ be the orthogonal matrix whose columns are $\eta_i(a)$ and $\tau_j(a)$, namely that
\begin{equation}\label{def:nitk}
\eta_i(a) = Q_ae_i \quad \text{and} \quad \tau_j(a) = Q_ae_j,
\end{equation}
and define 
\begin{equation}\label{def:wa}
w_a(y,s) = (T-t)^{\frac{1}{p-1}}u(x,t), \quad y = Q_a^T\left(\frac{x - a}{\sqrt{T-t}}\right), \quad s = -\log (T-t),
\end{equation}
then we see from \eqref{def:WaI} that $w_a$ satisfies \eqref{equ:WaI} and 
\begin{equation}\label{def:wabyWa}
w_a(y,s) = W_a(Q_ay,s), \quad \forall (y,s) \in \RN \times [-\log T, +\infty).
\end{equation}
Note from \eqref{def:nitk} that the point $(y,s)$ in the domain of $w_a$ becomes the point $(x,t)$ in the domain of $u$, where
\begin{equation*}
x = a + e^{-\frac s2}Q_ay = a + e^{-\frac s2}\left(\sum_{i = 1}^\ell y_i \eta_i(a) + \sum_{j = \ell + 1}^N y_j\tau_j(a)\right), \quad t = T - e^{-s}.
\end{equation*}
Now, fix $a \in S_\delta$ and consider an arbitrary $b \in S_\delta$. From \eqref{def:wa}, we have
\begin{equation}\label{equ:wa_wb}
w_a(y,s) = w_b(Y,s), \quad \text{where}\;\;\; Y = Q_b^T\left(Q_ay + e^{\frac s2}(a-b)\right).
\end{equation}
If we differentiate \eqref{equ:wa_wb} with respect to $y_k$ with $k \in \{\ell + 1, \cdots, N\}$, we get
\begin{align}
&(T-t)^{\frac{1}{p-1} + \frac{1}{2}}\frac{\partial u}{\partial \tau_k(a)}(x,t)\nonumber\\
&= \frac{\partial w_a}{\partial y_k}(y,s) = \sum_{i = 1}^\ell \tau_k(a)\cdot \eta_i(b)\frac{\partial w_b}{\partial y_i}(Y,s) + \sum_{j = \ell +1}^N\tau_k(a)\cdot\tau_j(b)\frac{\partial w_b}{\partial y_j}(Y,s).\label{equ:diffwa}
\end{align}
If we fix $b$ as the projection of $x = a + e^{-\frac s2}Q_ay$ on the blow-up set in the orthogonal direction to the tangent space to the blow-up set at $a$, then $b$ has the same components on the tangent space spanned by $\{\tau_{\ell + 1}(a), \cdots, \tau_N(a)\}$ as $x$. In particular, 
\begin{equation}\label{def:point_b}
b = b(a,y,s) = a + \sum_{i = 1}^\ell \gamma_{a,i}(e^{-\frac s2}\tilde{y})\eta_i(a) + \sum_{j = \ell + 1}^N e^{-\frac s2}y_j \tau_j(a), \quad \tilde{y} = (y_{\ell + 1}, \cdots, y_N).
\end{equation}
The following proposition gives a geometric constraint on the expansion of $w_a$, which is the bridge linking the refined asymptotic behavior to the refined regularity of the blow-up set.

\begin{proposition}[A geometric constraint on the expansion of $w_a$] \label{prop:geocon} Assume that 
$$\gamma_a \in \mathcal{C}^{1,\alpha^*}((-\epsilon_a, \epsilon_a)^{N -\ell}, \mathbb{R}^\ell) \quad \text{for some}\;\; \alpha^* \in \left(0, \frac 12\right)\;\; \text{and}\;\; \epsilon_a >0.$$
Then, there exists $s_1 \geq \max\{-\log T,s_0\}$ ($s_0$ is introduced in Proposition \ref{prop:expodecay}) such that for all $a \in S_\delta$, $|y| \leq 1$, $s \geq s_1$ and $k = \ell+1, \cdots, N$, it holds that
\begin{align}
&\left|\frac{\partial w_a}{\partial y_k}(y,s) - \left\{\frac{\partial w_b}{\partial y_k}(\bar{y}, 0, \cdots, 0, s) + \frac{\kappa}{2ps} \sum_{i = 1}^\ell \frac{\partial \gamma_{a,i}}{\partial \xi_k}(e^{-\frac s2} \tilde{y})y_i \right\}\right|\nonumber\\
& \leq C \sum_{i = 1}^\ell\left|\frac{\partial \gamma_{a,i}}{\partial \xi_k}(e^{-\frac s2} \tilde{y})\right| \left[|\bar{y}|\frac{\log s}{s^2} + \frac 1s e^{-\frac{\alpha^*s}{2}} + e^{-\frac s2}s^{C_0}\right] + C e^{-\frac{(1 + \alpha^*)s}{2}}s^{C_0},\label{est:geoconst}
\end{align}
where $\bar{y} = (y_1, \cdots, y_\ell)$, $\tilde{y} = (y_{\ell + 1}, \cdots, y_N)$ and $b$ is defined by \eqref{def:point_b}.
\end{proposition}
\begin{proof} Note that the proof of Proposition \ref{prop:geocon} was given in \cite{ZAAdm06} only when $\ell = 1$. Of course, that proof naturally extends to the case when $\ell \in \{2, \cdots, N-1\}$. Since our paper is relevant only when $\ell \geq 2$ and Proposition \ref{prop:geocon} presents an essential link between the asymptotic behavior of the solution and a geometric constraint of the blow up set, we felt we should give the proof of this proposition for the completeness and for the reader's convenience. As said earlier, this section just gives the main steps of the proof of Theorem \ref{theo:1}, and because the proof is long and technical, we leave it to Section \ref{sec:ap1}.
\end{proof}

\subsection{Part 3: Refined regularity of the blow-up set and conclusion of Theorem \ref{theo:1}.}
In this subsection, we give the proof of the $\mathcal{C}^2$-regularity of the blow-up set (Theorems \ref{theo:1} and \ref{theo:2}). We proceed in 2 steps:
\begin{itemize}
\item Step 1: We derive from Proposition \ref{prop:expodecay} that $\gamma_a$ is $\mathcal{C}^{1, \frac 12 - \eta}$ for all $\eta > 0$. Then we apply Proposition \ref{prop:geocon} with $\alpha^* = \alpha \in (0, \frac 12)$ to improve the regularity of $\gamma_a$ which reaches to $\mathcal{C}^{1, 1 - \eta}$ for all $\eta > 0$.
\item Step 2: Using the $\mathcal{C}^{1, 1-\eta}$-regularity and the geometric constraint in Proposition \ref{prop:geocon}, we refine the asymptotic behavior given in Proposition \ref{prop:expodecay}, which involves terms of order $\frac{1}{s}e^{-\frac s2}$. Exploiting this refined asymptotic behavior together with the geometric constraint \eqref{est:geoconst}, we derive that $\gamma_a$ is of class  $\mathcal{C}^2$, which is the conclusion of Theorem \ref{theo:1}. From the information obtained on the $\mathcal{C}^2$-regularity, we calculate the second fundamental form of the blow-up set, which concludes the proof of Theorem \ref{theo:2}.
\end{itemize}

\subsubsection*{Step 1: Deriving $\mathcal{C}^{1,1 - \eta}$-regularity of the blow-up set.}
We first derive the $\mathcal{C}^{1,\frac{1}{2} - \eta}$-regularity of the blow-up set for all $\eta > 0$ from Proposition \ref{prop:expodecay}. Then we apply Proposition \ref{prop:geocon} with $\alpha^* = \alpha \in (0,\frac 12)$ to get $\mathcal{C}^{1, 1 - \eta}$-regularity for all $\eta > 0$. In particular, we claim the following:

\begin{proposition}[$C^{1,\frac{1}{2} - \eta}$-regularity for $S$] \label{prop:C1alph} Under the hypotheses of Theorem \ref{theo:1}, $S$ is the graph of a vector function $\gamma \in \mathcal{C}^{1, \frac{1}{2} - \eta}((-\delta_1, \delta_1)^{N-\ell}, \mathbb{R}^\ell)$ for any $\eta > 0$, locally near $\hat{a}$. More precisely, there is an $h_0 > 0$ such that for all $|\tilde{\xi}| < \delta_1$ and $|\tilde{h}| < h_0$ such that $|\tilde{\xi} + \tilde{h}| < \delta_1$, we have for all $i \in \{1, \cdots, \ell\}$,
\begin{equation}\label{est:propC12al}
|\gamma_{i}(\tilde{\xi} + \tilde{h}) - \gamma_{i}(\tilde{\xi}) - \tilde{h}\cdot \nabla \gamma_{i}(\tilde{\xi})| \leq C|\tilde{h}|^{\frac 32}|\log |\tilde{h}||^{\frac 12 + \frac{C_0}{2}}.
\end{equation}
\end{proposition}

\begin{proof} The proof is mainly based on the derivation of the sharp asymptotic profile given in Proposition \ref{prop:expodecay}. In fact, we exploit the estimate \eqref{est:beyLogT_t} to find out a geometric constraint on the blow-up set $S$, which implies some more regularity on $S$. Since the argument follows the same lines as in Section 4, \cite{ZAAcmp02} for the case $\ell = 1$, and no new ideas are needed for the case $\ell \geq 2$, we will just sketch the proof by underlying the most relevant aspects in Section \ref{sec:C1alpha} for the reader's sake.
\end{proof}

The next proposition shows the $\mathcal{C}^{1,1-\eta}$-regularity of the blow-up set.
\begin{proposition}[$C^{1,1 - \eta}$-regularity for $S_\delta$] \label{prop:C11eta} There exists $\xi_0 > 0$ such that for each $a \in S_\delta$, the local chart defined in \eqref{def:localChart} satisfies for all $k = \ell+1, \cdots, N$ and $|\tilde{\xi}|< \xi_0$, 
$$\sum_{i = 1}^{\ell}\left|\frac{\partial \gamma_{a,i}}{\partial \xi_k}(\tilde{\xi}) \right| \leq C|\tilde{\xi}||\log |\tilde{\xi}||^{1 + \mu} \quad \text{for some}\;\; \mu > 0.$$
\end{proposition}
\begin{proof} Note that the case $\ell = 1$ was already proven in \cite{ZAAdm06} (see Lemma 3.4, page 516). Here we use again the argument of \cite{ZAAdm06} for the case $\ell \geq 2$. Using the estimate given in Proposition \ref{prop:expodecay} and parabolic regularity, we see that for all $k \geq \ell + 1$ and $s \geq s_0 +1$, 
$$\sup_{a \in S_\delta, |y| < 2} \left|\frac{\partial w_a}{\partial y_k}(y,s)\right| \leq Ce^{-\frac s2}s^{\mu} \quad \text{for some }\;\; \mu >0.$$
Consider $a \in S_\delta$ and $y = (\bar{y}, \tilde{y})$, where $\bar{y} = (y_1, \cdots, y_\ell)$ is such that $y_{i_*} = 1$ for some $i_* \in \{1, \cdots, \ell\}$, $y_j = 0$ for $1 \leq j \ne i_* \leq \ell$,  and $\tilde{y} = (y_{\ell +1}, \cdots, y_N)$ is arbitrary in $\partial B_{N-\ell}(0,1)$. For $s \geq \max\{s_0 + 1, s_1\}$, we consider $b = b(a,y,s)$ defined as in \eqref{def:point_b}. Since $\gamma_{a}$ is $\mathcal{C}^{1, \frac 12 - \eta}$ for any $\eta > 0$, we use \eqref{est:geoconst} with $\alpha^* = \alpha \in (0,\frac 12)$ to write for $k \in \{ \ell+1, \cdots, N\}$,
$$\frac{\kappa}{2ps} \left|\frac{\partial \gamma_{a,i_*}}{\partial \xi_k}(e^{-\frac s2}\tilde{y})\right| \leq C \frac{\log s}{s^2}\sum_{i = 1}^\ell \left|\frac{\partial \gamma_{a,i}}{\partial \xi_k}(e^{-\frac s2}\tilde{y})\right| + Ce^{-\frac s2}s^{\mu}.$$
Since $i_*$ is arbitrary in $\{1, \cdots, \ell\}$, we get 
$$\frac{\kappa}{2ps} \sum_{i = 1}^\ell\left|\frac{\partial \gamma_{a,i}}{\partial \xi_k}(e^{-\frac s2}\tilde{y})\right| \leq C \frac{\log s}{s^2}\sum_{i = 1}^\ell \left|\frac{\partial \gamma_{a,i}}{\partial \xi_k}(e^{-\frac s2}\tilde{y})\right| + Ce^{-\frac s2}s^{\mu},$$
which gives
$$\sum_{i = 1}^\ell \left|\frac{\partial \gamma_{a,i}}{\partial \xi_k}(e^{-\frac s2}\tilde{y})\right| \leq C e^{-\frac s2}s^{1 + \mu}.$$
If $\tilde{\xi} = e^{-\frac s2}\tilde{y}$, then $|\tilde{\xi}| = e^{-\frac s2}$ and $|\log |\tilde{\xi}|| = \frac s2$ since $|\tilde{y}| = 1$. Therefore, 
$$\sum_{i = 1}^\ell \left|\frac{\partial \gamma_{a,i}}{\partial \xi_k}(e^{-\frac s2}\tilde{y})\right| \leq C |\tilde{\xi}||\log|\tilde{\xi}||^{1 + \mu}.$$
Since $\tilde{y}$ is arbitrary in $\partial B_{N-\ell}(0,1)$, $\tilde{\xi} = e^{-\frac s2}\tilde{y}$ covers a whole neighborhood of $0$, namely $B(0, \xi_0)$ where $\xi_0 = e^{-\frac 12\max\{s_0+1,s_1\}}$, this concludes the proof of Proposition \ref{prop:C11eta}.
\end{proof}

\subsubsection*{Step 2: Further refined asymptotic behavior and deriving $\mathcal{C}^2$-regularity of $S$.}
In this part, we shall use the $\mathcal{C}^{1,1-\eta}$-regularity of the blow-up set together with the geometric constraint \eqref{est:geoconst} in order to refine further the asymptotic behavior \eqref{est:expo_s}. In particular, we claim the following:
\begin{proposition}[Further refined asymptotic behavior \eqref{est:expo_s}] \label{prop:furRe} There exist $s_2 > 0$, $d \in (0, \frac 12)$ and continuous functions $a \to \lambda_\beta(a)$ for all $\beta \in \mathbb{N}^N$ with $|\beta| = 3$ and $|\bar{\beta}| = 1$, where $\bar{\beta} = (\beta_1, \cdots, \beta_{\ell})$, $|\bar{\beta}| = \sum\limits_{i=1}^{\ell}\beta_i$, such that for all $a \in S_\delta$ and $s \geq s_2$,
\begin{equation}\label{est:furRe}
\left\|W_a(Q_ay,s) - w_{\mathcal{B}(a)}(\bar{y}, s) - \frac{e^{-\frac s2}}{s}\sum_{|\beta| = 3, |\bar{\beta}| = 1}\lambda_\beta(a)h_\beta(y)\right\|_{L^2_\rho} \leq Ce^{-\frac s2}s^{d-\frac 32},
\end{equation}
where $h_\beta$ is defined in \eqref{def:hbeta}.
 \end{proposition}
\begin{proof} The proof of this proposition is based on ideas of \cite{ZAAdm06} where the case $\ell = 1$ was treated. As in \cite{ZAAdm06}, the geometric constraint given in Proposition \ref{prop:geocon} plays an important role in deriving \eqref{est:furRe}. Since the proof is long and technical, we leave it to Section \ref{sec:futre}. 
\end{proof}

Let us derive Theorem \ref{theo:1} from Propositions \ref{prop:furRe} and \ref{prop:geocon}. In particular, Theorem \ref{theo:1} is a direct consequence of the following:
\begin{proposition}\label{prop:idenC2} For all $a \in S_\delta$, we have for all $i \in \{1, \cdots, \ell\}$, $j,k \in \{\ell + 1, \cdots, N\}$,
$$\Lambda^{(i)}_{j,k}(a) = \frac{\partial^2 \gamma_{a,i}}{\partial \xi_j \xi_k}(0) = \frac{2p}{\kappa}(1 + \delta_{j,k})\lambda_{e_i + e_j + e_k}(a),$$
where $a \to \lambda_{\beta}(a)$ is introduced in Proposition \ref{prop:furRe}, $e_i$ is the $i$-th vector of canonical base of $\RN$, and $\delta_{i,k}$ is the Kronecker symbol.
\end{proposition}
\begin{proof} From \eqref{def:wabyWa}, \eqref{est:furRe} and the fact that estimate  \eqref{est:furRe} also holds in $W^{2, \infty}(|y| < 2)$ by parabolic regularity, we derive for all $k \geq \ell + 1$ and $s \geq s_2 + 1$, 
\begin{equation}\label{est:thLM1}
\sup_{a \in S_\delta, |y|< 2} \left|\frac{\partial w_a}{\partial y_k}(y,s) - \frac{e^{-\frac s2}}{s}\sum_{|\beta| = 3, |\bar{\beta}| = 1} \lambda_\beta(a)\frac{\partial h_\beta}{\partial y_k}(y)\right| \leq Ce^{-\frac s2}s^{d - \frac 32},
\end{equation}
for some $d \in (0,\frac 12)$.\\
Note that if $|\bar{\beta}| = 1$, then there is a unique index $i^* \in \{1, \cdots, \ell\}$ such that $\beta_{i^*} = 1$ and $\beta_m = 0$ for $m\in \{1,\cdots, \ell\}$, $m \ne i^*$. Note also from the definition of $h_\beta$ (see \eqref{def:hbeta} below)that 
$$\frac{\partial h_\beta}{\partial y_k}(y) = \beta_k h_{\beta_k - 1}(y_k) \prod_{j = 1, j \ne k}^N h_{\beta_j}(y_j),$$
and that $h_0 = 1$. Therefore, \eqref{est:thLM1}  yields
$$\left|\frac{\partial w_a}{\partial y_k}(y,s) - \frac{e^{-\frac s2}}{s}\sum_{i = 1}^{\ell}\sum_{|\beta| = 3, \beta_{i} = 1} \lambda_\beta(a) h_{1}(y_{i})\beta_k h_{\beta_k - 1}(y_k) \prod_{j = \ell_a +1, j \ne k} h_{\beta_j}(y_j)\right| \leq Ce^{-\frac s2}s^{d - \frac 32}.$$
Take $i^* \in \{1, \cdots, \ell\}$ arbitrarily and $y = e_{i^*} + \epsilon e_j$ where $\epsilon = \pm 1$ and $j \geq \ell + 1$, and note that $h_{m}(0) = 0$ if $m$ is odd, and that if $|\beta| = 3, \beta_{i^*} = 1$, then either $\beta = e_{i^*} + e_{j^*} + e_{k^*}$ or $\beta = e_{i^*} + 2e_{j^*}$ for some $j^*, k^* \in \{\ell + 1, \cdots,N\}$, the above identity yields
\begin{equation}\label{est:proLM2}
\left|\frac{\partial w_a}{\partial y_k}(e_{i^*} + \epsilon e_j,s) - \epsilon \frac{e^{-\frac s2}}{s}(1 + \delta_{k,j})\lambda_{e_{i^*} + e_k + e_j}(a)\right| \leq Ce^{-\frac s2}s^{d - \frac 32}.
\end{equation}
Similarly, we have 
\begin{equation}\label{est:proLM3}
\left|\frac{\partial w_a}{\partial y_k}(e_{i^*},s)\right| \leq Ce^{-\frac s2}s^{d - \frac 32}.
\end{equation}
Now using Proposition \ref{prop:geocon}, we write for $y = e_{i^*} + \epsilon e_j$ and $s \geq \max\{s_2 + 1, s_1\}$,
\begin{align*}
&\left|\frac{\partial w_a}{\partial y_k}(e_{i^*} + \epsilon e_j,s) -  \frac{\partial w_a}{\partial y_k}(e_{i^*},s) - \frac{\kappa}{2ps} \frac{\partial \gamma_{a, i^*}}{\partial \xi_k}(e^{-\frac s2}\epsilon e_j)\right|\\
&\leq C\frac{\log s}{s^2} \sum_{i = 1}^{\ell} \left|\frac{\partial \gamma_{a, i}}{\partial \xi_k}(e^{-\frac s2}\epsilon e_j)\right| + Ce^{-\frac{(1 + \alpha^*)s}{2}}s^{C_0} + Ce^{-s}s^{C_0 + 1}.
\end{align*}
Using this estimate together with \eqref{est:proLM2} and \eqref{est:proLM3}, we obtain
\begin{align}
&\left|\epsilon e^{-\frac s2}(1 + \delta_{k,j})\lambda_{e_{i^*} + e_k + e_j}(a) - \frac{\kappa}{2p} \frac{\partial \gamma_{a, i^*}}{\partial \xi_k}(e^{-\frac s2}\epsilon e_j)\right|\nonumber\\
&\leq C\frac{\log s}{s} \sum_{i = 1}^{\ell} \left|\frac{\partial \gamma_{a, i}}{\partial \xi_k}(e^{-\frac s2}\epsilon e_j)\right| + Ce^{-\frac s2}s^{d - \frac 12}.\label{est:tmpC21}
\end{align}
From Proposition \ref{prop:furRe}, we see that 
$$\forall s \geq s_2, \quad \|W_a(Q_ay,s) - w_{\mathcal{B}(a)}(\bar{y},s)\|_{L^2_\rho} \leq Cs^{-1} e^{-\frac s2}.$$
Using this estimate and noticing that the same proof of Proposition \ref{prop:C11eta} holds with $\mu = -1$, we derive
$$\sum_{i = 1}^{\ell} \left|\frac{\partial \gamma_{a, i}}{\partial \xi_k}(e^{-\frac s2}\epsilon e_j)\right| \leq Ce^{-\frac s2}.$$
Putting this estimate into \eqref{est:tmpC21} and note that $\frac{\partial \gamma_{a, i^*}}{\partial \xi_k}(0) = 0$, we find that 
\begin{equation}\label{equ:idenC2}
\frac{\partial^2\gamma_{a, i^*}}{\partial \xi_k \partial \xi_j}(0) = \lim_{s \to +\infty} \frac{\frac{\partial \gamma_{a, i^*}}{\partial \xi_k}(e^{-\frac s2}\epsilon e_j)}{\epsilon e^{-\frac s2}} = \frac{2p}{\kappa}(1 + \delta_{k,j})\lambda_{e_{i^*} + e_k + e_j}(a).
\end{equation}
Since $i^*$ is taken arbitrarily belonging to $\{1, \cdots, \ell\}$, identity \eqref{equ:idenC2} holds for all $i^* \in \{1, \cdots, \ell\}$. This concludes the proof of Proposition \ref{prop:idenC2}. 
\end{proof}

\begin{proof}[Proof of Theorem \ref{theo:1}] From the definition of the local chart \eqref{def:localChart}, we have for all $i \in \{1, \cdots, \ell\}$, $\gamma_{a,i}(0) = \nabla \gamma_{a,i}(0) = 0$. Hence, we deduce from \eqref{equ:idenC2} the expression of the second fundamental form of the blow-up set at the point $a$ along the unitary basic vector $Q_a^Te_i$:  for all $k,j \in \{\ell+1, \cdots, N\}$, 
\begin{equation}\label{def:Lambdakj}
\Lambda_{k,j}^{(i)}(a) = \frac{\partial^2\gamma_{a, i}}{\partial \xi_k \partial \xi_j}(0) = \frac{2p}{\kappa}(1 + \delta_{k,j})\lambda_{e_{i} + e_k + e_j}(a).
\end{equation}
In addition, since $a \to \lambda_\beta(a)$ is continuous, we conclude that the blow-up set is of class $\mathcal{C}^2$. This completes the proof of Theorem \ref{theo:1}.
\end{proof}

\begin{proof}[Proof of Theorem \ref{theo:2}] The estimate \eqref{est:link} directly follows from Propositions \ref{prop:furRe} and \ref{prop:idenC2}. Indeed, the sum in estimate \eqref{est:furRe} can be indexed as 
$$\{\beta \in \mathbb{N}^N, |\beta| = 3, |\bar{\beta}| = 1\} = \{e_i + e_j + e_k, 1\leq i \leq \ell, \ell + 1\leq j, k\leq N\},$$
where $e_k$ is the $k$-th canonical basis vector of $\RN$. From \eqref{def:Lambdakj} and the definition of $h_\beta$ (see \eqref{def:hbeta} below), we write
\begin{align*}
\sum_{|\beta| = 3, |\bar{\beta}| = 1}\lambda_\beta(a)h_\beta(y) & = \sum_{i = 1}^{\ell} \sum_{j,k = \ell +1}^N \lambda_{e_i + e_j +e_k}h_{e_i+e_j+e_k}(y)\\
&= \frac{\kappa}{2p}\sum_{i = 1}^{\ell}y_i \sum_{j,k =\ell + 1}^N \frac{\Lambda^{(i)}_{j,k}(a)}{1 + \delta_{j,k}}(y_jy_k - 2\delta_{j,k}),
\end{align*}
which yields \eqref{est:link}. 

As for \eqref{est:limitLam}, we note from \eqref{est:furRe} that for all $|\beta| = 3$ with $|\bar{\beta}| = 1$ that (recall that $g_{a}(y,s) = W_a(Q_ay,s) - w_{\mathcal{B}(a)}(\bar{y},s)$), 
$$\left|g_{a, \beta}(s) - \frac{e^{-\frac s2}}{s}\lambda_{\beta}(s)\right| \leq Ce^{-\frac s2}s^{d - \frac 32}.$$
Hence, we write from \eqref{def:Lambdakj}, 
\begin{align*}
\Lambda^{(i)}_{j,k}(a) &= \frac{2p}{\kappa}(1 + \delta_{j,k})\lambda_{e_i + e_j +e_k}(a)\\
&=\frac{2p}{\kappa}(1 + \delta_{j,k})\lim_{s\to +\infty}se^{\frac s2}g_{a, e_i + e_j +e_k}(s)\\
&= \frac{2p}{\kappa}(1 + \delta_{j,k})\lim_{s\to +\infty}se^{\frac s2}\int_{\RN}g_{a}(y,s)\frac{h_{e_i + e_j +e_k}(y)}{\|h_{e_i + e_j +e_k}\|^2_{L^2_\rho}}\rho(y)dy
\end{align*}
Using again the definition of $h_\beta$ (see \eqref{def:hbeta} below), we see that $h_{e_i + e_j +e_k} = y_i(y_jy_k - \delta_{j,k})$ and $\|h_{e_i + e_j +e_k}\|^2_{L^2_\rho} = 8(1 + \delta_{j,k})$. Recall that $w_{\mathcal{A}}$ does not depend on $y_j$ for $j \geq \ell + 1$. Hence, for all $j,k \geq \ell + 1$,
$$\Lambda^{(i)}_{j,k}(a) = \frac{p}{4\kappa}\lim_{s\to +\infty}se^{\frac s2}\int_{\RN}W_{a}(Q_ay,s)y_i(y_jy_k - 2\delta_{j,k}) \rho(y)dy,$$
which is \eqref{est:limitLam}. This concludes the proof of Theorem \ref{theo:2}.
\end{proof}

\section{Proof of Propositions \ref{prop:class}, \ref{prop:geocon}, \ref{prop:C1alph} and \ref{prop:furRe}.}\label{sec:proofAll}
\subsection{Classification of the difference of two solutions of \eqref{equ:WaI} having the same asymptotic behavior.}\label{sec:3}
In this subsection, we give the proof of Proposition \ref{prop:class}. The formulation is the same as given in \cite{FZnon00} for the difference of two solutions with the radial profile ($\ell = N$). Therefore, we sketch the proof and emphasize only on the novelties. Note also that the case $\ell = 1$ was treated in \cite{ZAAcmp02}.

Let us define 
\begin{equation}\label{def:ga}
g(y,s) = W_1(y,s) - W_2(y,s), 
\end{equation}
where $W_i,\; i = 1,2$ are the solutions of equation \eqref{equ:WaI} and behave like \eqref{hyp:W1W2}.  We see from \eqref{equ:WaI} and \eqref{hyp:W1W2} that for all $(y,s) \in \RN \times [-\log T, +\infty)$, 
\begin{equation}\label{equ:ga}
\left\{\begin{array}{l}
\partial_s g = \Lc  g + \alpha g,\\
\|g(s)\|_{L^2_\rho} \leq C\frac{\log s}{s^2},
\end{array}\right.
\end{equation}
where 
$$\Lc  = \Delta -\frac{1}{2}y\cdot \nabla + 1,$$
and 
\begin{equation*}
\alpha(y,s) = \frac{|W_1|^{p-1}W_1 - |W_2|^{p-1}W_2}{W_1 - W_2} - \frac{p}{p-1} \quad \text{if} \;\; W_1 \ne W_2,
\end{equation*}
in particular,
\begin{equation}\label{def:alpha2}
\alpha(y,s) = p|W_0(y,s)|^{p-1} - \frac{p}{p-1}, \quad \text{for some}\;\; W_0(y,s) \in (W_1(y,s), W_2(y,s)).
\end{equation}

The operator $\Lc $ is self-adjoint on $\mathcal{D}(\Lc ) \subset L^2_\rho(\RN)$. Its spectrum consists of eigenvalues
$$spec(\Lc ) = \left\{\lambda_n = 1 - \frac{n}{2},\; n \in \N\right\}.$$
The eigenfunctions corresponding to $1 - \frac{n}{2}$ are
\begin{equation}\label{def:hbeta}
h_\beta(y) = h_{\beta_1}(y_1)\cdots h_{\beta_N}(y_N),\quad \beta_1 + \cdots + \beta_N = |\beta| = n,
\end{equation}
where 
\begin{equation*}
h_m(\xi) = \sum_{i = 0}^{[m/2]} \frac{m!}{i!(m - 2i)!}(-1)^i\xi^{m - 2i} \quad \text{for}\;\; m \in \N,
\end{equation*}
satisfy 
\begin{equation*}
\int_\R h_m(\xi)h_n(\xi)\rho(\xi)d\xi = 2^m m! \delta_{m,n}.
\end{equation*}
The component of $g$ on $h_\beta$ is given by
\begin{equation*}
g_{\beta}(s) = \int_{\RN} k_\beta(y)g(y,s)\rho(y)dy, \quad \text{where}\quad k_\beta(y) = \frac{h_\beta(y)}{\|h_\beta\|_{L^2_\rho}^{2}}.
\end{equation*}
If we denote by $P_n$ the orthogonal projector of $L^2_\rho$ over the eigenspace of $\Lc $  corresponding to the eigenvalue $1 - \frac{n}{2}$, then
\begin{equation*}
P_ng(y,s) = \sum_{|\beta| = n} g_{\beta}(s)h_\beta(y).
\end{equation*}
Since the eigeinfunctions of $\Lc $ span the whole space $L^2_\rho$, we can write
\begin{equation*}
g(y,s) = \sum_{n \in N}P_ng(y,s) = \sum_{\beta \in \N^N}g_{\beta}(s)h_\beta(y) = \sum_{\beta \in \N^N, |\beta| \leq k}g_{\beta}(s)h_\beta(y) + R_{k+1}g(y,s),
\end{equation*}
where $R_kg = \sum\limits_{n\geq k} P_ng$. We also denote
\begin{equation}\label{def:Ia}
I(s)^2 = \|g(s)\|_{L^2_\rho}^2 = \sum_{n \in \N}l_{n}^2(s) = \sum_{n \leq k}l_{n}^2(s) + r_{k+1}^2(s),
\end{equation}
where 
\begin{equation}\label{def:lnrk}
l_{n}(s) = \|P_ng(s)\|_{L^2_\rho}, \quad r_{k}(s) = \|R_kg(s)\|_{L^2_\rho}.
\end{equation}

As for $\alpha$, we have the following estimates:
\begin{lemma}[Estimates on $\alpha$] \label{lemm:estalpha} For all $y \in \RN$ and $s \geq -\log T$, we have
$$\alpha(y,s) \leq \frac{C}{s},\quad |\alpha(y,s)| \leq \frac{C}{s}(1 + |y|^2),$$
and 
\begin{equation}\label{equ:mainalpha}
\left|\alpha(y,s) + \frac{1}{4s}\sum_{i = 1}^{\ell}h_2(y_i) \right| \leq \frac{C}{s^{\frac 32}}(1 + |y|^3).
\end{equation}
\end{lemma}
\begin{proof} The proof follows the same lines as the proof of Lemma 2.5 in \cite{FZnon00} where the case $\ell = N$ was treated.  
\end{proof}

In the following lemma, we project equation \eqref{equ:ga} on the different modes  to get estimates for $I(s)$, $l_{n}(s)$ and $r_{n}(s)$. More precisely, we claim the following:
\begin{lemma}[Evolution of $I(s)$, $l_{n}(s)$ and $r_{n}(s)$]\label{lemm:evoIlr} There exist $s_3 \geq -\log T$ and $s_* > 0$ such that for all $s \geq s_3$, $n \in \N$ and $\beta \in \N^N$, we have
\begin{itemize}
\item[i)] $\left|l'_{n}(s) + \left(\frac{n}{2} - 1\right)l_{n}(s)\right| \leq C(n)\dfrac{I(s)}{s}$.
\item[ii)] $I'(s) \leq \left(1 - \frac{n+1}{2} + \frac{C_0}{s}\right)I(s) + \sum\limits_{k = 0}^n \frac{1}{2}(n + 1 - k)l_{k}(s)$.
\item[iii)] $\left|g'_{\beta}(s) + \left(-1 + \frac{|\beta|}{2} + \frac{1}{s}\sum\limits_{i=1}^{\ell}\beta_i \right)g_{\beta}(s)\right| \leq C(\beta)\left(\dfrac{1}{s^{\frac 32}}I(s) + \dfrac{1}{s}\left(l_{|\beta| - 2}(s) + l_{|\beta| + 2}\right)\right)$.
\item [iv)] $r'_{n}(s) \leq \left(1 - \frac{n}{2}\right)r_{n}(s) + \frac{C}{s}I(s - s_*)$.
\end{itemize}
\end{lemma}
\begin{proof} For $(i)$ and $(ii)$, see Lemma 2.7, page 1197 in \cite{FZnon00}. For $(iii)$, see Appendix B.1, page 545 in \cite{ZAAcmp02} for a similar calculations. For $(iv)$, see page 523 in \cite{ZAAdm06}, where the calculation is mainly based on the following regularizing property of equation \eqref{equ:ga} by Herrero and Vel\'azquez \cite{HVaihn93} (control of the $L^4_\rho$-norm by the $L^2_\rho$-norm up to some delay in time, see Lemma 2.3 in \cite{HVaihn93}):
\begin{equation*}
\left(\int g^4(y,s)\rho dy \right)^{1/4} \leq C\left(\int g^2(y,s - s_*)\rho dy \right)^{\frac 12} \quad \text{for some}\;\; s_* > 0.
\end{equation*}
This ends the proof of Lemma \ref{lemm:evoIlr}.
\end{proof}

In the next step, we use Lemma \ref{lemm:evoIlr} to show that either the null mode or a negative mode of $\Lc $ will dominate as $s \to +\infty$. In particular, we have the following:
\begin{proposition}[Dominance of a mode and its description]\label{prop:dom} We have\\
$i)\;$ Either for all $n \in \N$, $l_{n}(s) = \mathcal{O}\left(\dfrac{I(s)}{s}\right)$ and there exist $\sigma_n$, $C_n >0$ and $C_n' > 0$ such that 
$$\forall s \geq \sigma_n, \quad I(s) \leq C_ns^{C'_n}exp \left((1 - n/2)s\right).$$ 
$ii)$ Or there is $n_0 \geq 2$ such that 
\begin{equation}\label{eq:Iln0}
I(s) \sim l_{n_0}(s)\;\;\text{and}\;\; \forall n \ne n_0, \;\; l_{n}(s) = \mathcal{O}\left(\dfrac{I(s)}{s}\right) \quad \text{as}\;\;s \to +\infty.
\end{equation}
Moreover, 
\begin{itemize}
\item[$\bullet$] If $n_0 = 2$, namely $I(s) \sim l_{2}(s)$,  then 
\begin{equation}\label{equ:ODEbetaa2}
\forall |\beta| = 2, \quad \left\{\begin{array}{ll}
|g_{\beta}(s)| \leq C\frac{\log s}{s^{5/2}} &\text{if}\;\; \sum_{i=1}^{\ell} \beta_i \ne 2,\\
\left|g_{\beta}(s) - \frac{c_{\beta}}{s^2}\right| \leq C\frac{\log s}{s^{5/2}} &\text{if}\;\; \sum_{i=1}^{\ell} \beta_i = 2.
\end{array} \right.
\end{equation}
\item[$\bullet$] If $n_0 = 3$, namely $I(s) \sim l_{3}(s)$,  then 
\begin{equation}\label{est:In03}
I(s) \leq C_0e^{-\frac s2}s^{C_0} \quad \text{for some}\;\; C_0 > 0.
\end{equation}
\end{itemize}
\end{proposition}
\begin{proof} See Proposition 2.6, page 1196 in \cite{FZnon00} for the existence of a dominating component, where the proof relies on $(i)$ and $(ii)$ of Lemma \ref{lemm:evoIlr}. If case $(ii)$ occurs with $n_0 = 2$, we write from $(iii)$ of Lemma \ref{lemm:evoIlr}: for all $\beta \in \mathbb{N}^N$ with $|\beta| = 2$,
$$\left|g'_{\beta}(s) + \frac{g_{\beta}}{s}\sum_{i = 1}^{\ell}\beta_i\right| \leq C(\beta) \left(\dfrac{I(s)}{s^{\frac 32}}+ \dfrac{l_{0}(s) + l_{4}(s)}{s}\right) \leq C(\beta)\dfrac{I(s)}{s^{\frac 32}} \leq C(\beta)\frac{\log s}{s^{7/2}},$$
where we used \eqref{eq:Iln0} and \eqref{equ:ga} from which $l_{0}(s) + l_{4}(s) = \mathcal{O}\left(\frac{I(s)}{s}\right)$ and $I(s) = \mathcal{O}\left(\frac{\log s}{s^2}\right)$. Since $\sum\limits_{i = 1}^{\ell} \beta_i$ is only equal to $0, 1$ or $2$ if $|\beta| = 2$, estimate \eqref{equ:ODEbetaa2} follows after integration. Estimate \eqref{est:In03} immediately follows from $(i)$ of Lemma \ref{lemm:evoIlr}. This ends the proof of Proposition \ref{prop:dom}. 
\end{proof}

Let us now derive Proposition \ref{prop:class} from Proposition \ref{prop:dom}. Indeed, we see from Proposition \ref{prop:dom} that if case $i)$ occurs, we already have exponential decay for $I(s)$. If case $(ii)$ occurs with $n_0 \geq 3$, we write from part $(i)$ of Lemma \ref{lemm:evoIlr}, 
$$\left| l'_{n_0}(s) + \left(\frac{n_0}{2} - 1\right)l_{n_0}\right| \leq \frac{C}{s}l_{n_0}.$$
Since $l_{n_0} \ne 0$ in a neighbourhood of infinity, this gives 
$$l_{n_0}(s) \leq C_0 s^{C_0}e^{\left(1 - \frac{n_0}{2}\right)s} \leq C_0s^{C_0} e^{-\frac{s}{2}},$$
which yields \eqref{equ:case2class}. If case $(ii)$ occurs with $n_0 = 2$, by definition of $P_2$, we derive from \eqref{equ:ODEbetaa2} that there is a symmetric, real $(\ell \times \ell)$-matrix $\mathcal{B}$ such that
$$P_2g(y,s) = \frac{1}{s^2}\left(\frac{1}{2}\bar{y}^T\mathcal{B}\bar{y} - tr(\mathcal{B})\right) + o\left(\frac{1}{s^2}\right),$$
which is \eqref{equ:case1class}. This concludes the proof of Proposition \ref{prop:class}.\hfill $\square$

\subsection{$\mathcal{C}^{1,\frac{1}{2} - \eta}$-regularity of the blow-up set.} \label{sec:C1alpha}
We give the proof of Proposition \ref{prop:C1alph} in this section. The proof uses the argument given in \cite{ZAAcmp02} treated for the case $\ell = 1$. Here we shall exploit the refined estimate \eqref{est:beyLogT_t} to obtain a geometric constraint on the blow-up set. Without loss of generality, we assume $\hat{a} = 0$ and $Q_{\hat{a}} = Id$. Under the hypotheses of Proposition \ref{prop:C1alph}, we know that $\gamma \in \mathcal{C}^1((-\delta_1, \delta_1)^{N - \ell}, \mathbb{R}^\ell)$ with $\ell \in \{1, \cdots, N-1\}$. If we introduce
$$\Gamma(\tilde{x}) = (\gamma_1(\tilde{x}), \cdots, \gamma_\ell(\tilde{x}),\tilde{x}), \quad \tilde{x} = (x_{\ell + 1}, \cdots, x_N),$$
then 
$$\text{Im}\,\Gamma \cap B(0, 2\delta) = \text{graph}(\gamma) \cap B(0, 2\delta) = S_\delta.$$
Consider $\tilde{x}$ and $\tilde{h}$ in $\mathbb{R}^{N-\ell}$ such that $\tilde{x}$ as well as $\tilde{x} + \tilde{h}$ are in $B(0,\delta_1)$ and $\Gamma(\tilde{x})$ as well as $\Gamma(\tilde{x} + \tilde{h})$ are in $S_\delta$. For all $t\in [T-e^{-s_0 - 1}, T)$ such that $|\Gamma(\tilde{x}) - \Gamma(\tilde{x} + \tilde{h})| \leq \sqrt{(T-t)|\log (T-t)|}$, we use \eqref{est:beyLogT_t} with $x = a = \Gamma(\tilde{x} + \tilde{h})$, then with $x = \Gamma(\tilde{x} + \tilde{h})$ and $a = \Gamma(\tilde{x})$) to find that
\begin{equation}\label{est:GxGh1}
\left\{
\begin{array}{cl}
\left|(T-t)^\frac{1}{p-1}u(\Gamma(\tilde{x} + \tilde{h}),t) - w_{\mathcal{B}(\Gamma(\tilde{x} + \tilde{h}))}(0, s)\right|& \leq Ce^{-\frac s2}s^{\frac 32 + C_0},\\
\left|(T-t)^\frac{1}{p-1}u(\Gamma(\tilde{x} + \tilde{h}),t) - w_{\mathcal{B}(\Gamma(\tilde{x}))}(\bar{y}_{\Gamma(\tilde{x}), \Gamma(\tilde{x} + \tilde{h}),s}, s)\right|& \leq Ce^{-\frac s2}s^{\frac 32 + C_0},
\end{array}
\right.
\end{equation}
where $\bar{y}_{\Gamma(\tilde{x}), \Gamma(\tilde{x} + \tilde{h}),s}$ is defined as
\begin{equation}\label{def:barya1a2}
\bar{y}_{a_1, a_2,s} = e^{\frac s2}\big((a_1 - a_2)\cdot Q_{a_1}e_1, \cdots, (a_1 - a_2)\cdot Q_{a_1}e_\ell\big).
\end{equation}
Since $\Gamma$ is $\mathcal{C}^1$, we have 
$$|\Gamma(\tilde{x} + \tilde{h}) - \Gamma(\tilde{x})| \leq C|\tilde{h}|.$$
Let us fix $t = \tilde{t}(\tilde{x}, \tilde{h})$ such that 
\begin{equation}\label{rel:tildet}
|\Gamma(\tilde{x} + \tilde{h}) - \Gamma(\tilde{x})| = \sqrt{(T-\tilde{t})|\log (T - \tilde{t})|},
\end{equation}
and take $\tilde{h} \in B_{N-\ell}(0, h_1(s_0))$ for some $h_1(s_0)>0$, then we have $\tilde{t} \geq T - e^{-s_0-1}$. Hence, if $\tilde{s} = - \log(T-\tilde{t})$, we have by \eqref{est:GxGh1},
\begin{equation}\label{est:ID1}
\left|w_{\mathcal{B}(\Gamma(\tilde{x} + \tilde{h}))}(0, \tilde{s}) -  w_{\mathcal{B}(\Gamma(\tilde{x}))}(\bar{y}_{\Gamma(\tilde{x}), \Gamma(\tilde{x} + \tilde{h}), \tilde{s}}, \tilde{s})\right| \leq Ce^{-\frac{\tilde{s}}{2}}\tilde{s}^{\frac 32 + C_0}.
\end{equation}
Similarly, by changing the roles of $\tilde{x}$ and $\tilde{x} + \tilde{h}$, we get
\begin{equation}\label{est:ID2}
\left|w_{\mathcal{B}(\Gamma(\tilde{x}))}(0, \tilde{s}) -  w_{\mathcal{B}(\Gamma(\tilde{x} + \tilde{h}))}(\bar{y}_{\Gamma(\tilde{x} + \tilde{h}), \Gamma(\tilde{x}),\tilde{s}}, \tilde{s})\right| \leq Ce^{-\frac{\tilde{s}}{2}}\tilde{s}^{\frac 32 + C_0},
\end{equation}
where $\bar{y}_{\Gamma(\tilde{x} + \tilde{h}), \Gamma(\tilde{x}),\tilde{s}}$ is defined as in \eqref{def:barya1a2}.\\

From a Taylor expansion for $w_{\mathcal{B}}(\bar{y},\tilde{s})$ near $\bar y = 0$, we write
\begin{equation}\label{eq:TwB}
w_{\mathcal{B}}(\bar{y}, \tilde{s}) = w_{\mathcal{B}}(0, \tilde{s}) + \bar{y}\cdot \nabla w_{\mathcal{B}}(0, \tilde{s}) + \frac{1}{2}\bar{y}^T\nabla^2w_{\mathcal{B}}(0, \tilde{s})\bar{y} + \mathcal{O}\left(|\bar{y}|^3 |\nabla^3 w_{\mathcal{B}}(z, \tilde{s})|\right),
\end{equation}
for some $z$ between $0$ and  $\bar{y}$. 

\noindent Since \eqref{est:what} and \eqref{est:whatwA} also hold in $\mathcal{C}_{loc}^k$ by parabolic regularity, we deduce that
$$|\nabla w_{\mathcal{B}}(0,\tilde{s})| = \mathcal{O}\left(\frac{\log \tilde{s}}{\tilde{s}^2}\right), \quad \nabla^2 w_{\mathcal{B}}(0, \tilde{s}) = - \frac{\kappa}{4p \tilde{s}}I_{\ell \times \ell} + \mathcal{O}\left(\frac{\log \tilde{s}}{\tilde{s}^2}\right).$$
From \cite{MZgfa98} (see Theorem 1), we know that $\|\nabla^3 w_{\mathcal{B}}(\tilde{s})\|_{L^\infty} \leq \frac{C_3}{\tilde{s}^{\frac 32}}$. Substituting all these above estimates into \eqref{eq:TwB} yields
$$w_{\mathcal{B}}(\bar{y}, \tilde{s}) \leq w_{\mathcal{B}}(0, \tilde{s}) - \frac{\kappa}{8p \tilde{s}}|\bar{y}|^2 + \frac{C_3|\bar{y}|^3}{6 \tilde{s}^{\frac 32}} + \frac{C\log \tilde{s}}{\tilde{s}^2}.$$
Therefore, we have
\begin{equation}\label{est:ID5}
 \forall |\bar{y}| \leq \frac{3\kappa}{8C_3p}\sqrt{\tilde{s}}, \quad w_{\mathcal{B}}(\bar{y}, \tilde{s}) \leq w_{\mathcal{B}}(0, \tilde{s}) - \frac{\kappa}{16p \tilde{s}}|\bar{y}|^2.
\end{equation}

We claim from \eqref{est:ID1}, \eqref{est:ID2} and \eqref{est:ID5} the following:
\begin{equation}\label{est:ID3}
\left|w_{\mathcal{B}(\Gamma(\tilde{x}))}(0, \tilde{s}) -  w_{\mathcal{B}(\Gamma(\tilde{x} + \tilde{h}))}(0, \tilde{s})\right| \leq Ce^{-\frac{\tilde{s}}{2}}\tilde{s}^{\frac 32 + C_0},
\end{equation}
Indeed, if $w_{\mathcal{B}(\Gamma(\tilde{x}))}(0, \tilde{s}) -  w_{\mathcal{B}(\Gamma(\tilde{x} + \tilde{h}))}(0, \tilde{s}) \geq 0$, then we have by \eqref{est:ID5} and \eqref{est:ID2},
\begin{align*}
0 &\leq w_{\mathcal{B}(\Gamma(\tilde{x}))}(0, \tilde{s}) -  w_{\mathcal{B}(\Gamma(\tilde{x} + \tilde{h}))}(0, \tilde{s})\\
&\leq w_{\mathcal{B}(\Gamma(\tilde{x}))}(0, \tilde{s}) - w_{\mathcal{B}(\Gamma(\tilde{x} + \tilde{h}))}(\bar{y}_{\Gamma(\tilde{x} + \tilde{h}), \Gamma(\tilde{x}),\tilde{s}}, \tilde{s}) \leq Ce^{-\frac{\tilde{s}}{2}}\tilde{s}^{\frac 32 + C_0}.
\end{align*}
If $w_{\mathcal{B}(\Gamma(\tilde{x}))}(0, \tilde{s}) -  w_{\mathcal{B}(\Gamma(\tilde{x} + \tilde{h}))}(0, \tilde{s}) \leq 0$, then we do as above and use \eqref{est:ID1} instead of \eqref{est:ID2} to obtain \eqref{est:ID3}.\\
From \eqref{est:ID3}, \eqref{est:ID1} and \eqref{est:ID5}, we get 
\begin{equation*}
\frac{\kappa}{16p \tilde{s}}|\bar{y}_{\Gamma(\tilde{x}), \Gamma(\tilde{x} + \tilde{h}), \tilde{s}}|^2 \leq w_{\mathcal{B}(\Gamma(\tilde{x}))}(0, \tilde{s}) - w_{\mathcal{B}(\Gamma(\tilde{x}))}(\bar{y}_{\Gamma(\tilde{x}), \Gamma(\tilde{x} + \tilde{h}), \tilde{s}}, \tilde{s}) \leq Ce^{-\frac{\tilde{s}}{2}}\tilde{s}^{\frac 32 + C_0}.
\end{equation*}
Hence, we obtain 
\begin{equation}\label{est:ID6}
|\bar{y}_{\Gamma(\tilde{x}), \Gamma(\tilde{x} + \tilde{h}), \tilde{s}}|^2 \leq Ce^{-\frac{\tilde{s}}{2}}\tilde{s}^{\frac 52 + C_0}.
\end{equation}
From the definition \eqref{def:barya1a2}, we have 
\begin{equation}\label{est:ID7}
|\bar{y}_{\Gamma(\tilde{x}), \Gamma(\tilde{x} + \tilde{h}), \tilde{s}}| = e^{\frac{\tilde{s}}{2}}d(\Gamma(\tilde{x}), \pi_{\Gamma(\tilde{x} + \tilde{h})}),
\end{equation}
where we recall $\pi_{\Gamma(\tilde{x} + \tilde{h})}$ is the tangent plan of $S$ at $\Gamma(\tilde{x} + \tilde{h})$. In the other hand, we claim that
\begin{equation}\label{est:ID8}
d(\Gamma(\tilde{x}), T_{\Gamma(\tilde{x} + \tilde{h})}) \geq \frac{|\gamma_i(\tilde{x} + \tilde{h}) - \gamma_i(\tilde{x}) - \tilde{h}\cdot \nabla \gamma_i(\tilde{x})|}{\sqrt{1 + |\nabla \gamma_i(\tilde{x})|^2}},
\end{equation}
where $S_i$ is the surface of equation $x_i = \gamma_i(\tilde{x})$,  $T_{i, \Gamma(\tilde{x} + \tilde{h})}$ is the tangent plan of $S_i$ at $\Gamma(\tilde{x} + \tilde{h})$. Indeed, we note that 
$$d(\Gamma(\tilde{x}), T_{i,\Gamma(\tilde{x} + \tilde{h})}) = \frac{|\gamma_i(\tilde{x} + \tilde{h}) - \gamma_i(\tilde{x}) - \tilde{h}\cdot \nabla \gamma_i(\tilde{x})|}{\sqrt{1 + |\nabla \gamma_i(\tilde{x})|^2}},$$
and $Im\, \Gamma \subset S_i$, hence, \eqref{est:ID8} follows from $d(\Gamma(\tilde{x}), T_{\Gamma(\tilde{x} + \tilde{h})}) \geq d(\Gamma(\tilde{x}), T_{i,\Gamma(\tilde{x} + \tilde{h})})$.

Combining \eqref{est:ID6}, \eqref{est:ID7}, \eqref{est:ID8} together with the relation $\tilde{s} = -\log (T - \tilde{t})$ yields
$$|\gamma_i(\tilde{x} + \tilde{h}) - \gamma_i(\tilde{x}) - \tilde{h}\cdot \nabla \gamma_i(\tilde{x})|^2 \leq C(T - \tilde{t})^{\frac 32}|\log (T - \tilde{t})|^{\frac 52 + C_0}.$$
If we denote $A = |\Gamma(\tilde{x} + \tilde{h}) - \Gamma(\tilde{x})| \leq C|\tilde{h}|$, then we have  by  the relation \eqref{rel:tildet},
$$ |\log (T - \tilde{t})| \sim 2 |\log A|, \quad T - \tilde{t} \sim \frac{A^2}{2|\log A|} \quad \text{as} \quad A \to 0.$$
Hence,
$$|\gamma_i(\tilde{x} + \tilde{h}) - \gamma_i(\tilde{x}) - \tilde{h}\cdot \nabla \gamma_i(\tilde{x})|^2 \leq C A^3 |\log A|^{1 + C_0} \leq C |\tilde{h}|^3 |\log |\tilde{h}||^{1 + C_0},
$$
which yields \eqref{est:propC12al}. This concludes the proof of Proposition \ref{prop:C1alph}.\hfill $\square$

\subsection{A geometric constraint linking the blow-up behavior of the solution to the regularity of the blow-up set.} \label{sec:ap1}
This section is devoted to the proof of Proposition \ref{prop:geocon}. The proof follows ideas given in \cite{ZAAdm06}. Recall from the hypothesis that $\gamma_a \in \mathcal{C}^{1, \alpha^*}((-\epsilon_a, \epsilon_a)^{N-\ell}, \mathbb{R}^\ell)$ for some $\alpha^* \in (0, \frac 12)$ and $\epsilon_a > 0$, and that $\gamma_{a,i}(0) = \nabla \gamma_{a,i}(0) = 0$, we have for all $|\tilde{\xi}| < \epsilon_a$, 
\begin{equation}\label{est:varphiC1alp}
|\gamma_{a,i}(\tilde{\xi})| \leq C|\tilde{\xi}|^{1 + \alpha^*} \quad \text{and} \quad |\nabla \gamma_{a,i}(\tilde{\xi})| \leq C|\tilde{\xi}|^{\alpha^*}.
\end{equation}
In what follows, $k \in \{\ell + 1, \cdots, N\}$ is fixed, and we use indexes $i$ and $m$ for the range $1, \cdots, \ell$, index $j$ for the range $\ell +1, \cdots, N$. 

We now use \eqref{est:varphiC1alp} to approximate all the terms appearing in \eqref{equ:diffwa}.

$(a)$ Term $\tau_k(a) \cdot \eta_i(b)$. From the local coordinates \eqref{def:point_b}, we have 
\begin{equation*}
\eta_i(b) = \frac{1}{\sqrt{1 + |\nabla \gamma_{a,i}(e^{-\frac s2} \tilde{y})|^2}}\left(\eta_i(a) - \sum_{j = \ell + 1}^N \frac{\partial \gamma_{a,i}}{\partial \xi_{j}}(e^{-\frac s2}\tilde{y}) \tau_j(a)\right).
\end{equation*}
Using \eqref{est:varphiC1alp} and the fact that $\tau_k(a)\cdot \eta_i(a) = 0$ and $\tau_k(a)\cdot \tau_j(a) = \delta_{k,j}$, we obtain
\begin{align}
\left|\tau_k(a)\cdot \eta_i(b) + \frac{\partial \gamma_{a,i}}{\partial \xi_{k}}(e^{-\frac s2}\tilde{y})\right| & = \left| \left(1 -\frac{1}{\sqrt{1 + |\nabla \gamma_{a,i}(e^{-\frac s2} \tilde{y})|^2}}\right) \frac{\partial \gamma_{a,i}}{\partial \xi_{k}}(e^{-\frac s2}\tilde{y})\right| \nonumber\\
& \leq  \left|\frac{\partial \gamma_{a,i}}{\partial \xi_{k}}(e^{-\frac s2}\tilde{y})\right| |\nabla \gamma_{a,i}(e^{-\frac s2}\tilde{y})|^2 \nonumber\\
& \leq \left|\frac{\partial \gamma_{a,i}}{\partial \xi_{k}}(e^{-\frac s2}\tilde{y})\right| e^{-\alpha^*s}.\label{est:tauk_etai}
\end{align}

$(b)$ Term $\tau_k(a)\cdot \tau_j(b)$. From \eqref{def:point_b} and \eqref{est:varphiC1alp}, we have 
$$|b - a| \leq \left|\sum_{i = 1}^\ell \gamma_{a,i}(e^{-\frac s2}\tilde{y})\right| + e^{-\frac s2}|\tilde{y}| \leq C e^{-\frac s2}.$$
Since $\eta_i$ and $\tau_j$ are $\mathcal{C}^{\alpha^*}$, it holds that 
$$|\eta_i(a) - \eta_i(b)| + |\tau_j(a) - \tau_j(b)| \leq C|a - b|^{\alpha^*} \leq C e^{-\alpha^* \frac s2}.$$
This follows that 
\begin{equation}\label{est:difftangnor}
\begin{array}{c}
|\eta_i(a)\cdot\eta_m(b) - \delta_{i,m}| + |\tau_k(a)\cdot\tau_j(b) - \delta_{k,j}| \leq Ce^{-\alpha^*\frac s2},\\
|\eta_i(a)\cdot \tau_j(b)| + |\eta_i(b)\cdot \tau_j(a)| \leq Ce^{-\alpha^*\frac s2}.
\end{array}
\end{equation}

$(c)$ The point $Y(a, y, s)$. Using \eqref{def:nitk}, \eqref{equ:wa_wb} and \eqref{def:point_b}, we write
\begin{align*}
Y_m &= Y\cdot e_m = \big(Q_ay + e^{\frac s2}(a - b) \big)\cdot Q_be_m\\
& = \left\{\sum_{i = 1}^\ell y_i \eta_i(a) + \sum_{j = \ell +1}^Ny_j \tau_j(a) - e^{\frac s2} \left[\sum_{i = 1}^\ell \gamma_{a,i}(e^{-\frac s2}\tilde{y})\eta_i(a) + \sum_{j = \ell + 1}^Ne^{-\frac s2}y_j \tau_j(a)\right] \right\}\cdot Q_be_m\\
& = \left\{\sum_{i = 1}^\ell \big[y_i - e^{\frac s2}\gamma_{a,i}(e^{-\frac s2}\tilde{y})\big]\eta_i(a)\right\}\cdot Q_be_m.
\end{align*}
From \eqref{def:nitk}, we write for $m \in \{1, \cdots, \ell\}$,
\begin{align*}
Y_m - y_m &= \left\{ \left(y_m - e^{\frac s2}\gamma_{a,m}(e^{-\frac s2}\tilde{y})\right)\eta_m(a)\cdot \eta_m(b) - y_m \eta_m(a)\cdot \eta_m(a)\right\}\\
& + \sum_{i = 1, i \ne m}^\ell \left(y_i - e^{\frac s2}\gamma_{a,i}(e^{-\frac s2}\tilde{y})\right)\eta_i(a)\cdot \eta_m(b),
\end{align*}
and for $n \in \{\ell+1, \cdots, N\}$,
$$Y_n = \sum_{i = 1}^\ell \left(y_i - e^{\frac s2}\gamma_{a,i}(e^{-\frac s2}\tilde{y})\right)\eta_i(a)\cdot \tau_n(b).$$
Using \eqref{est:difftangnor} yields
$$|Y_m - y_m| \leq Ce^{-\alpha^*\frac s2} \quad \text{and}\quad |Y_k| \leq Ce^{-\alpha^*\frac s2}.$$
Hence, if we write 
$$\bar{Y} = (Y_1, \cdots, Y_\ell)\quad \text{and} \quad \tilde{Y} = (Y_{\ell +1}, \cdots, Y_N),$$
then
\begin{equation}\label{est:Yhat_yhat}
|\bar{y} - \bar{Y}| \leq Ce^{-\alpha^*\frac s2} \quad \text{and}\quad |\tilde{Y}| \leq Ce^{-\alpha^*\frac s2}.
\end{equation}

$(d)$ Term $\frac{\partial w_b}{\partial y_i}(Y,s)$. From Proposition \ref{prop:expodecay} and the parabolic regularity, we have that
\begin{equation}\label{est:wa_WAexpoW2loc}
\sup_{s \geq s'}\left\|w_b(y,s) - w_{\mathcal{B}(b)}(\bar{y},s)\right\|_{W_{loc}^{2,\infty}(|\bar{y}| < 2)} \leq Ce^{-\frac s2}s^{C_0}.
\end{equation}
This implies 
\begin{align}
&\left|\frac{\partial w_b}{\partial y_i}(Y,s) - \frac{\partial w_{\mathcal{B}(b)}}{\partial y_i}(\bar{y},s)\right| + \sum_{m = \ell + 1}^N \left|\frac{\partial w_b}{\partial y_m}(Y,s)\right|\nonumber\\
& + \sup_{|z| < 2, (m,n) \ne (i,i), i \geq \ell + 1} \left| \frac{\partial^2 w_b}{\partial y_m \partial y_n}(z,s)\right| \leq Ce^{-\frac s2}s^{C_0}.\label{est:wbwAbC2}
\end{align}
Similarly, from \eqref{est:Wa} and \eqref{def:wabyWa}, 
\begin{equation}\label{est:w_alogW2loc}
\sup_{s \geq -\log T}\left\|w_a(y,s) - \left\{\kappa + \frac{\kappa}{2ps}\left(\ell - \frac{|\bar{y}|^2}{2}\right)\right\}\right\|_{W_{loc}^{2,\infty}(|\bar{y}| < 2)} \leq C\frac{\log s}{s^2}.
\end{equation}
From \eqref{est:wa_WAexpoW2loc} and \eqref{est:w_alogW2loc}, we deduce that
\begin{equation}\label{est:WAlogW2loc}
\sup_{s \geq s''}\left\|w_{\mathcal{B}(a)}(y,s) - \left\{\kappa + \frac{\kappa}{2ps}\left(\ell - \frac{|\bar{y}|^2}{2}\right)\right\}\right\|_{W_{loc}^{2,\infty}(|\bar{y}| < 2)} \leq C\frac{\log s}{s^2}.
\end{equation}
Using \eqref{est:WAlogW2loc}, we have for $|z| \leq 2$,
$$\left|\frac{\partial^2w_{\mathcal{B}(b)}}{\partial y_i^2}(z,s) + \frac{\kappa}{2ps}\right| \leq C\frac{\log s}{s^2} \quad \text{and} \quad \left|\frac{\partial^2w_{\mathcal{B}(b)}}{\partial y_i \partial y_m}(z,s)\right| \leq C\frac{\log s}{s^2}, \;\; m \ne i.$$
Note that $\frac{\partial w_{\mathcal{B}(b)}}{\partial y_i}(0,s) = 0$,  we then take the Taylor expansion of $\frac{\partial w_{\mathcal{B}(b)}}{\partial y_i}(\bar{y},s)$ near $\bar{y} = 0$ up to the first order to get 
$$\left|\frac{\partial w_{\mathcal{B}(b)}}{\partial y_i}(\bar{y},s) + Y_i\frac{\kappa}{2ps}\right| \leq C|\bar{y}|\frac{\log s}{s^2}.$$
Using \eqref{est:wbwAbC2} and \eqref{est:Yhat_yhat} yields
\begin{equation}\label{est:parwbyi}
\left|\frac{\partial w_b}{\partial y_i}(Y,s) + y_i\frac{\kappa}{2ps}\right| \leq Ce^{-\frac s2}s^{C_0} + C|\bar{y}|\frac{\log s}{s^2} + \frac{C}{s}e^{-\alpha^*\frac s2}.
\end{equation}

$(e)$ Term $\frac{\partial w_b}{\partial y_j}(Y,s)$. We just use \eqref{est:wbwAbC2} and \eqref{est:Yhat_yhat} to get
\begin{equation}\label{est:parwbyj}
\left|\frac{\partial w_b}{\partial y_j}(Y,s) - \frac{\partial w_b}{\partial y_j}(\bar{y},0, \cdots, 0, s)\right| \leq Ce^{-(1 + \alpha^*)\frac s2}.
\end{equation}

Estimate \eqref{est:geoconst} then follows by substituting \eqref{est:parwbyj}, \eqref{est:parwbyi},  \eqref{est:wbwAbC2}, \eqref{est:tauk_etai} and \eqref{est:difftangnor} into \eqref{equ:diffwa}. This concludes the proof of Proposition \ref{prop:geocon}.\hfill $\square$

\subsection{Further refined asymptotic behavior.}\label{sec:futre}
We prove Proposition \ref{prop:furRe} in this subsection. We first refine estimate \eqref{est:expo_s} and find following terms in the expansion which is of order $e^{-\frac s2}$. Using the geometric constraint, we show that all terms of order $e^{-\frac s2}$ must be identically zero, which gives a better estimate for $\|W_a(Q_ay,s) - w_{\mathcal{B}(a)}(\bar{y}, s)\|_{L^2_\rho}$. We then repeat the process and use again Proposition \ref{prop:geocon} in order to get the term of order $\frac 1s e^{-\frac s2}$ and conclude the proof of Proposition \ref{prop:furRe}. 

Let us define 
\begin{equation}\label{def:gaWawB}
g_a(y,s) = W_a(Q_ay,s) - w_{\mathcal{B}(a)}(\bar{y},s),
\end{equation}
and denote by 
$$I_a(s)^2 = \|g_a(s)\|_{L^2_\rho}^2, \quad l_{a,n}(s) = \|P_n g_a(s)\|_{L^2_\rho}, \quad r_{a,k}(s) = \|\sum_{n \geq k} P_n g_a(s)\|_{L^2_\rho}.$$
From \eqref{est:expo_s}, we have
\begin{equation}\label{est:Iarecal}
I_a(s) = \mathcal{O} \left(e^{-\frac s2}s^{\mu}\right) \;\; \text{for some}\;\; \mu >0.
\end{equation}
Note that Lemma \ref{lemm:evoIlr} also holds with $W_1 = W_a$ and $W_2 = w_{\mathcal{B}}$. We claim the following:
\begin{lemma}\label{lemm:equbeta3} Assume that $I_a(s) = \mathcal{O} \left(e^{-\frac s2}s^{\mu_0}\right)$ for some $\mu_0 \in \mathbb{R}$. There exists $s_4 > 0$ such that for all $s \geq s_4$,
\begin{equation}\label{est:l012r4}
\sum_{n = 0}^2 l_{a,n}(s) + r_{a,4}(s) \leq Ce^{-\frac s2}s^{\mu_0 - 1},
\end{equation}
and 
\begin{equation}\label{equ:beta3}
\forall \beta \in \mathbb{N}^N, \; |\beta| = 3, \quad \left|\frac{d}{ds} \left(g_{a, \beta}(s)e^{\frac s2}s^{|\bar{\beta}|}\right)\right| \leq C s^{|\bar{\beta}| + \mu_0 - \frac 32},
\end{equation}
where $\bar{\beta} = (\beta_1, \cdots, \beta_{\ell})$, $|\bar{\beta}| = \sum \limits_{i=1}^{\ell}\beta_i$.
\end{lemma}
\begin{proof} From $(i)$ and $(iv)$ of Lemma \ref{lemm:evoIlr}, we write for all $s \geq s_3$, 
$$n =0, 1, 2, \;\;\left|\frac{d}{ds}\left(l_{a,n}(s)e^{(n/2-1)s}\right)\right| \leq Ce^{(n/2 - \frac 32)s} s^{\mu_0 - 1}, $$
and $$\left|\frac{d}{ds}\left(r_{a,4}(s)e^s\right)\right| \leq Ce^{\frac s2}s^{\mu_0 - 1}.$$
The estimate \eqref{est:l012r4} then follows after integration of the above inequalities . As for \eqref{equ:beta3}, we just use part $(iii)$ of Lemma \ref{lemm:evoIlr} and \eqref{est:l012r4} (note that $l_{a,5} \leq r_{a,4}$ by definition \eqref{def:lnrk}). This ends the proof of Lemma \ref{lemm:equbeta3}.
\end{proof}

Using \eqref{est:Iarecal} and applying Lemma \ref{lemm:equbeta3} a finite number of steps, we obtain the following:
\begin{lemma}\label{lemm:re_es12} There exist $s_5 > 0$ and continuous functions $a \to \lambda_\beta(a)$ for all $\beta \in \mathbb{N}^N$ with $|\beta| = 3$ and $|\bar{\beta}| = \sum\limits_{i = 1}^{\ell}\beta_i = 0$ such that for all $a \in S_\delta$ and $s \geq s_5$ , 
$$\left\|g_a(y,s) - e^{-\frac s2}\sum_{|\beta| = 3, |\bar{\beta}| = 0}\lambda_\beta(a)h_\beta(y) \right\|_{L^2_\rho} \leq Ce^{-\frac s2}s^{d - \frac 12},$$
for some $d \in \left(0, \frac 12\right)$, where  $h_\beta$ is defined by \eqref{def:hbeta}.
\end{lemma}
\begin{proof} We first show that there is $s_5 > 0$ such that
\begin{equation}\label{est:Iad12}
\forall s \geq s_5,\quad I_a(s) \leq Ce^{-\frac s2} s^d \quad \text{for some }\;\; d \in (0,\frac 12).
\end{equation}
From \eqref{est:Iarecal}, if $\mu \in (0, \frac 12)$, we are done. If $\mu \geq \frac 12$, we apply Lemma \ref{lemm:equbeta3} with $\mu_0 = \mu$ to get
$$\sum_{n = 0}^2 l_{a,n}(s) + r_{a,4}(s) \leq Ce^{-\frac s2}s^{\mu - 1},$$
and 
$$\forall |\beta| = 3, \quad |g_{a, \beta}(s)| \leq Ce^{-\frac s2}s^{\mu - \frac 12}.$$
Hence,
$$I_a(s) \leq Ce^{-\frac s2}s^{\mu - \frac 12}.$$
Estimate \eqref{est:Iad12} then follows by repeating this process a finite number of steps.

Now using \eqref{est:Iad12} and Lemma \ref{lemm:equbeta3} with $\mu_0 = d$,\\
- if $|\beta| = 3$ and $|\bar{\beta}| \geq 1$, we integrate \eqref{equ:beta3} on $[s, +\infty)$ to derive
$$\forall |\beta| = 3, \, |\bar{\beta}| \geq 1,\quad |g_{a, \beta}(s)| \leq Ce^{-\frac s2}s^{d - \frac 12},$$
- if $|\beta| = 3$ and $|\bar{\beta}| = 0$, by integrating \eqref{equ:beta3} on $[s_5, s]$, we deduce that  there exists continuous functions $a \to \lambda_\beta(a)$ such that 
$$\forall |\beta| = 3, \, |\bar{\beta}| = 0,\quad |g_{a, \beta}(s) - \lambda_\beta(a)e^{-\frac s2}| \leq Ce^{-\frac s2}s^{d - \frac 12}.$$
This concludes the proof of Lemma \ref{lemm:re_es12}.
\end{proof}

Now we shall use the geometric constraint on the asymptotic behavior of the solution
given in Proposition \ref{prop:geocon} to show that all the coefficients $\lambda_\beta(a)$ with $|\beta| = 3$ and $\bar{\beta} = 0$ in Lemma \ref{lemm:re_es12} have to be identically zero. In particular, we claim the following:
\begin{lemma} \label{lemm:exps12sd} There exists $s_6 > 0$ such that for all $s \geq s_6$,
$$\forall a \in S_\delta, \quad \|g_a(s)\|_{L^2_\rho} \leq Ce^{-\frac s2}s^{d - \frac 12} \quad \text{for some }\;\; d \in (0, \frac 12).$$
\end{lemma}
\begin{proof} Consider $a \in S_\delta$, we aim at proving that 
$$\forall \beta \in \mathbb{N}^N, \, |\beta| = 3, \, |\bar{\beta}| = 0, \quad \lambda_\beta(a) = 0,$$
where $\lambda_\beta(a)$ is introduced in Lemma \ref{lemm:re_es12} and $|\bar{\beta}| = \sum_{i = 1}^{\ell}\beta_i$.

From \eqref{def:wabyWa}, \eqref{def:gaWawB} and the fact that the estimate given in Lemma \ref{lemm:re_es12} also holds in $W^{2,\infty}(|y|<2)$ by parabolic regularity, we write for all $k \geq \ell + 1$ and $s \geq s_5 + 1$, 
\begin{equation}\label{est:lm1}
\sup_{a \in S_\delta, |y|< 2} \left|\frac{\partial w_a}{\partial y_k}(y,s) - e^{-\frac s2}\sum_{|\beta| = 3, \bar{\beta} = 0} \lambda_\beta(a)\frac{\partial h_\beta}{\partial y_k}(y)\right| \leq Ce^{-\frac s2}s^{d - \frac 12}.
\end{equation}
Take $y = (\bar{y}, \tilde{y})$, where $\bar{y} = (y_1, \cdots, y_\ell) = (0, \cdots, 0)$ and $\tilde{y} \in B_{N-\ell}(0, 1)$, then use Proposition \ref{prop:C11eta} and \eqref{est:geoconst}, we obtain
\begin{equation}\label{est:lm2}
\left|\frac{\partial w_a}{\partial y_k}(y,s) - \frac{\partial w_b}{\partial y_k}(0,s) \right| \leq Ce^{-(1 + \alpha^*)\frac s2}s^{C_0} + Ce^{-s}s^{C_0 + 1},
\end{equation}
for some $\alpha^* \in (0, \frac 12)$.

From \eqref{est:lm1} and \eqref{est:lm2}, we get
\begin{equation}\label{eq:tm22}
\left|\sum_{|\beta| = 3, |\bar{\beta}| = 0} \lambda_\beta(a)\frac{\partial h_\beta}{\partial y_k}(y) - \sum_{|\beta| = 3, |\bar{\beta}| = 0} \lambda_\beta(b)\frac{\partial h_\beta}{\partial y_k}(0)\right| \leq Cs^{d - \frac 12}.
\end{equation}
From \eqref{def:point_b} and Proposition \ref{prop:C11eta}, we see that $b \to a$ as $s \to +\infty$. Since $a \to \lambda_\beta(a)$ is continuous, $d \in (0,\frac 12)$, $h_{\beta_1}(0) = \cdots = h_{\beta_\ell}(0) = h_0(0) = 1$ from definition \eqref{def:hbeta}, and
$$\frac{\partial h_\beta}{\partial y_k}(y) = \beta_k h_{\beta_k - 1}(y_k) \prod_{j = 1, j \ne k}^N h_{\beta_j}(y_j),$$
we derive by passing to the limit in \eqref{eq:tm22},
\begin{align*}
&\sum_{|\beta| = 3, |\bar{\beta}| = 0} \lambda_\beta(a)\beta_kh_{\beta_k - 1}(y_k)\prod_{j = \ell + 1, j \ne k}^N h_{\beta_j}(y_j)\\
& \qquad= \sum_{|\beta| = 3, |\bar{\beta}| = 0} \lambda_\beta(a)\beta_kh_{\beta_k - 1}(0)\prod_{j = \ell + 1, j \ne k}^N h_{\beta_j}(0).
\end{align*}
By the orthogonality of the polynomials $h_i$, this yields
$$\beta_k \lambda_\beta(a) = 0, \quad \forall k \geq \ell+1,\; \forall |\beta|=3 \; \text{with}\; |\bar{\beta}| = 0.$$ 
Take $\beta$ arbitrary with $|\beta| = 3$ and $|\bar{\beta}| = 0$, then there exists $k \geq \ell + 1$ such that $\beta_k \geq 1$, which implies that $\lambda_\beta(a) = 0$. This ends the proof of Lemma \ref{lemm:exps12sd}.
\end{proof}

Let us now give the proof of Proposition \ref{prop:furRe} from Lemmas \ref{lemm:exps12sd} and \ref{lemm:equbeta3}. 
\begin{proof}[Proof of Proposition \ref{prop:furRe}] From Lemmas \ref{lemm:exps12sd} and \ref{lemm:equbeta3}, we see that for all $s \geq s_7 = \max\{s_4,s_5, s_6\}$, 
$$\sum_{n = 0}^2 l_{a,n}(s) + r_{a,4}(s) \leq Cs^{-\frac s2}s^{d - \frac 32},$$
and 
\begin{equation*}
\forall |\beta| = 3, \quad \left|\frac{d}{ds} \left(g_{a,\beta}(s)s^{\frac s2}s^{|\bar{\beta}|}\right)\right| \leq Ce^{|\bar{\beta}| + d - 2},
\end{equation*}
for some $d \in (0,\frac 12)$. Integrating this equation between $s$ and $+\infty$ if $|\bar{\beta}| = 0$ and between $s_7$ and $s$ if $|\bar{\beta}| \geq 1$, we get 
$$\forall |\beta| = 3,\quad |g_{a, \beta}(s)| \leq Ce^{-\frac s2}s^{d - 1}.$$
Hence, 
$$\forall s \geq s_7, \quad I_a(s) = \|g_a(s)\|_{L^2_\rho} \leq Ce^{-\frac s2}s^{d - 1}.$$
With this new estimate, we use again Lemma \ref{lemm:equbeta3} with $\mu_0 = d-1$ to show that there exists $s_8 > 0$ such that for all $s \geq s_8$, 
$$\sum_{n = 0}^2 l_{a,n}(s) + r_{a,4}(s) \leq Ce^{-\frac s2}s^{d - 2},$$
and 
\begin{equation*}
\forall |\beta| = 3, \quad \left|\frac{d}{ds} \left(g_{a,\beta}(s)e^{\frac s2}s^{|\bar{\beta}|}\right)\right| \leq Cs^{|\bar{\beta}| + d - \frac 52}.
\end{equation*}
This new equation implies that for all $|\beta| = 3$ and $s \geq s_8$, \\
- if $|\bar{\beta}| = 0$ or $|\bar{\beta}| \geq 2$, we have $|g_{a, \beta}(s)| \leq Ce^{-\frac s2}s^{d - \frac 32}$,\\
- if $|\bar{\beta}| = 1$, we obtain the existence of continuous functions $a \to \lambda_\beta(a)$ such that 
\begin{equation}\label{equ:gabeta3Th2}
\left|g_{a,\beta}(s) - \frac{e^{-\frac s2}}{s}\lambda_\beta(a)\right| \leq Ce^{-\frac s2}s^{d - \frac 32}.
\end{equation}
This concludes the proof of Proposition \ref{prop:furRe}.
\end{proof}

%\appendix
%\renewcommand*{\thesection}{\Alph{section}}
%\counterwithin{theorem}{section}

\bibliographystyle{alpha}
%\bibliography{D:/PhD/Travail/mybib}
%\bibliography{/Volumes/Data/Work/mybib} %for Mac

\def\cprime{$'$}

\bigskip
\end{document}